\documentclass[draft,oneside,reqno]{amsart}
\usepackage{amssymb,latexsym} 
\usepackage{verbatim} 
\usepackage{eucal} 

\theoremstyle{plain}
\newtheorem{theorem}{Theorem}[subsection]
\newtheorem{lemma}[theorem]{Lemma}
\newtheorem{corollary}[theorem]{Corollary}
\newtheorem{proposition}[theorem]{Proposition}

\theoremstyle{definition}

\newtheorem{convention}[theorem]{Convention}
\newtheorem{definition}[theorem]{Definition}
\newtheorem{notation}[theorem]{Notation}

\newtheorem{example}[theorem]{Example}
\newtheorem{remark}[theorem]{Remark}
\newtheorem{construction}[theorem]{Construction}

\DeclareMathOperator\ad{ad}
\DeclareMathOperator\Ad{Ad}

\DeclareMathOperator\Aut{Aut}
\DeclareMathOperator\cfo{cfo}

\DeclareMathOperator\diag{diag}
\DeclareMathOperator\GL{GL}

\DeclareMathOperator\gs{gs}
\DeclareMathOperator\End{End}
\DeclareMathOperator\Hom{Hom}
\DeclareMathOperator\id{id}
\DeclareMathOperator\Id{I}
\DeclareMathOperator\Lie{Lie}
\DeclareMathOperator\LT{LT} 
\DeclareMathOperator\Mat{Mat} 
\DeclareMathOperator\mg{mg} 

\DeclareMathOperator\rank{rank}
\DeclareMathOperator\SO{SO}
\DeclareMathOperator\spann{span}

\DeclareMathOperator\supp{supp}
\DeclareMathOperator\TKK{TKK}

\newcommand\andd{\quad \text{and} \quad}
\newcommand\form{(\, ,\, )}
\newcommand\orr{\quad \text{or} \quad}
\newcommand\order[1]{\left\vert #1 \right\vert}
\newcommand\ot{\otimes}
\newcommand \set[1]{\{#1 \}}
\newcommand\suchthat{\mid}

\newcommand \ft {\mathfrak t}
\newcommand \fg {\mathfrak g}
\newcommand \fh {\mathfrak h}
\newcommand \fn {\mathfrak n}
\newcommand \fs {\mathfrak s}

\newcommand\cA{\mathcal A}
\newcommand\cB{\mathcal B}

\newcommand\cE{{\mathcal E}}
\newcommand\cH{{\mathcal H}}

\newcommand\cJ{{\mathcal J}}

\newcommand\cL{{\mathcal L}}
\newcommand\cM{{\mathcal M}}

\newcommand\cO{{\mathcal O}}

\newcommand\cX{{\mathcal X}}

\newcommand{\al}{\alpha}
\newcommand{\ep}{\varepsilon}
\newcommand{\gm}{\gamma}
\newcommand{\Gm}{\Gamma}
\newcommand{\lm}{\lambda}
\newcommand{\Lm}{\Lambda}
\newcommand{\ph}{\varphi}
\newcommand{\sg}{\sigma}

\newcommand\hfg{{\hat{\mathfrak g}}}
\newcommand\hm{{\hat m}}
\newcommand\hsg{{\hat\sigma}}

\newcommand\modell{{\bar{\ell}}}
\newcommand{\blm}{{\bar\lambda}}
\newcommand{\bLm}{{\bar\Lambda}}
\newcommand\modz{{\bar 0}}

\newcommand\bbC{\mathbb C}
\newcommand\bbQ{\mathbb Q}
\newcommand\bbbQ{{\bar\bbQ}}
\newcommand\bbZ{\mathbb Z}
\newcommand\Zn{{\mathbb Z^n}}

\newcommand\boldm{{\mathbf m}}

\newcommand\boldone{{\boldsymbol{1}}}

\newcommand\bsg{{\boldsymbol{\sg}}}
\newcommand\btau{{\boldsymbol{\tau}}}

\newcommand\dCL[2]{C(\cL)_{#1}^{#2}}

\newcommand\dL[2]{\cL_{#1}^{#2}}
\newcommand\dLp[2]{{\cL'}_{#1}^{#2}}
\newcommand\ds[2]{\fs_{#1}^{#2}}

\newcommand{\De}{\Delta}
\newcommand{\Dec}{{\Delta^\times}}
\newcommand{\Deen}{{\Delta_{\mathrm{en}}}}
\newcommand{\Dep}{{\Delta_{+}}}
\newcommand{\Dei}{{\Delta_{\mathrm{ind}}}}
\newcommand{\Deip}{{\Delta'_{\mathrm{ind}}}}
\newcommand{\Deic}{{\Delta_{\mathrm{ind}}^\times}}
\newcommand{\Desh}{{\De_{\mathrm{sh}}}}

\newcommand\loopm{M_\boldm}
\newcommand\loopmp{M_{\boldm'}}

\newcommand\cLs{\cL^{(s)}}
\newcommand\sem{\Lm}
\newcommand\shom{s}
\newcommand\sltwo{\operatorname{sl}_2}
\newcommand\tbsg{{\tilde\bsg}}

\newcommand\tsg{\tilde \sigma}
\newcommand\tL{\widetilde{\cL}}

\begin{document}

\title[Multiloop realization of EALAs and Lie tori]{Multiloop realization of extended affine Lie algebras and Lie tori}
\author{Bruce Allison}
\address[Bruce Allison]
{Department of Mathematical and Statistical Sciences\\ University of
Alberta\\Edmonton, Alberta, Canada T6G 2G1}
\email{ballison@ualberta.ca}
\author{Stephen Berman}
\address[Stephen Berman]
{Saskatoon, Saskatchewan, Canada}
\email{sberman@shaw.ca}
\author{John Faulkner}
\address[John Faulkner]
{Department of Mathematics\\
University of Virginia\\
Kerchof Hall, P.O.~Box 400137\\
Charlottesville VA 22904-4137 USA}
\email{jrf@virginia.edu}
\author{Arturo Pianzola}
\address[Arturo Pianzola]
{Department of Mathematical and Statistical Sciences\\ University
of Alberta\\Edmonton, Alberta, Canada T6G 2G1}
\email{a.pianzola@ualberta.ca}
\thanks{The authors Allison and Pianzola gratefully acknowledge the support
of the Natural Sciences and Engineering Research Council of Canada.}
\date{September 6, 2007}

\begin{abstract}
An important theorem in the  theory of infinite dimensional Lie algebras states that any affine
Kac-Moody algebra
can be realized (that is to say constructed explicitly) using loop algebras.
In this paper, we consider the corresponding problem for a class
of Lie algebras called extended affine Lie algebras (EALAs) that generalize affine algebras.
EALAs occur in families that are constructed from centreless Lie tori,
so the realization problem for EALAs reduces to the realization problem for
centreless Lie tori.  We show that all but one
family of centreless Lie tori can be realized using multiloop algebras (in place of loop algebras).
We also obtain necessary and sufficient for two centreless Lie tori
realized in this way to be isotopic, a relation that corresponds to isomorphism
of the corresponding families of EALAs.
\end{abstract}

\maketitle

An \emph{extended affine Lie algebra}
(EALA) over a field of characteristic zero consists  of a Lie algebra $\cE$,
together with an nondegenerate invariant symmetric bilinear form $\form$ on $\cE$,
and a nonzero finite dimensional ad-diagonalizable  subalgebra $\cH$ of $\cE$,
such that a list of natural axioms are imposed (see \cite{N2} and the references
therein).  (Although, by definition, an EALA consists of a triple $(\cE,\cH, \form)$,
we usually abbreviate it as $\cE$.)
One of the axioms states that the group generated by the isotropic roots of $\cE$ is a
free abelian group $\Lm$ of finite rank, and the rank of $\Lm$ is called the \emph{nullity} of $\cE$.
As the term EALA suggests, the defining axioms for an EALA
are modeled after the properties of affine Kac-Moody Lie algebras; and in fact affine Kac-Moody Lie algebras
are precisely the extended affine Lie algebras of nullity 1.
So it is natural to look for a realization theorem for EALAs of arbitrary nullity $\ge 1$.
(Nullity 0 EALAs are finite dimensional simple Lie algebras
and we do not consider them in this context.)

The classical procedure for realizing affine Lie algebras using loop algebras proceeds in two
steps \cite[Chap.~7 and 8]{K}.
In the first step, the derived algebra modulo its centre of the affine algebra
is constructed as the loop algebra
of a diagram automorphism of a finite dimensional simple Lie algebra.
This loop algebra is naturally graded by $Q\times \bbZ$, where $Q$ is the root
lattice of a finite irreducible (but not necessarily reduced) root system.
In the second step, the affine algebra itself, together with a Cartan subalgebra and a nondegenerate invariant
bilinear form for the affine algebra, is built from the graded loop algebra by forming a central extension (with one dimensional centre) and adding a (one dimensional) graded algebra of derivations.

The replacement for the derived algebra modulo its centre in EALA theory is the centreless
core of the EALA, and centreless cores of EALAs have been characterized
axiomatically as centreless Lie tori \cite{Y2,N1}.
A \emph{Lie $\Lm$-torus}, or
a Lie torus for short, is a $Q\times \Lm$-graded Lie algebra $\cL$ satisfying a simple list of axioms,
where $Q$ is the root lattice of an irreducible finite root system $\De$ and
$\Lm$ is a free abelian group of finite rank (see subsection \ref{subsec:Lietoribasic}).
The type of the root system $\De$ is then called the \emph{type} of $\cL$
and the rank of $\Lm$ is called the \emph{nullity} of $\cL$.
A \emph{centreless Lie torus}
is a Lie torus with trivial centre.

Starting with a centerless Lie torus,
E.~Neher has shown how to carry out the second step of the classical realization
procedure \cite{N2}.
Specifically, given a centreless Lie torus~$\cL$ of nullity $n$, Neher gave a
simple construction of a family $\set{E(\cL,D,\tau)}_{(D,\tau)}$ of EALAs of nullity $n$
parameterized by  pairs $(D,\tau)$, where $D$ is a graded
Lie algebra of derivations of $\cL$ (subject to some
additional restrictions) and $\tau$ is a graded invariant
2-cocycle with values in the graded dual $C$ of~$D$.
(As a vector space $E(\cL,D,\tau) = \cL \oplus C \oplus D$,
and $\tau$ is used in the description of the multiplication of two elements of $D$.  For the largest
possible choice of $D$, the algebra of skew-centroidal derivations of $\cL$,
$\cL \oplus C$ is the universal central extension of $\cL$.)
Moreover he announced that every  EALA $\cE$
is isomorphic as an EALA to $E(\cL,D,\tau)$ for some $\cL$ and
$(D,\tau)$ \cite[Thm~4.2]{N2}.  (Two EALAs are isomorphic if there is a Lie algebra isomorphism from one to the other
preserving the given forms up to nonzero scalar and preserving the given ad-diagonalizable subalgebras.)

Thus, only the first step in the realization procedure needs to be considered---the realization
problem for centreless Lie tori.  If $k$ is algebraically closed and $n\ge 1$,
we describe a construction
of a centreless Lie $\Zn$-torus $\LT(\fs,\bsg,\fh)$, called a \emph{multiloop Lie $\Zn$-torus},
starting from a finite dimensional simple Lie algebra~$\fs$, a sequence
$\bsg$ of $n$-commuting finite order automorphisms satisfying three conditions labeled as
(A1)-(A3), and a Cartan subalgebra $\fh$ of the fixed point algebra $\fs^\bsg$ of $\bsg$ is $\fs$.
In our first main theorem, which we call the \emph{realization theorem} (see Theorem \ref{thm:realization}),
we prove that a centreless Lie torus $\cL$ of nullity $\ge 1$ is bi-isomorphic
(defined below) to a multiloop Lie torus if and only if $\cL$ is fgc (that is,  $\cL$ is
finitely generated as a module over its centroid).

We note that when $n=1$ the conditions (A1)--(A3) on a single automorphism $\sg$
are equivalent to the assumption that $\sg$ is a diagram automorphism (see subsection \ref{subsec:n=1}),
thus relating our realization theorem to the classical realization theorem for affine algebras.

Our realization theorem is a satisfying answer to the realization problem for centreless Lie tori,
and hence for EALAs, since it is known from work of Neher \cite{N1} that there is only one family of centreless Lie tori
that are not fgc, namely the special linear
Lie algebras coordinatized by quantum tori whose unit group has
infinite commutator group (see Remark \ref{rem:almostallfgc}).
Of course by the classical realization theorem (or directly),
we know that these exceptions do not exist in nullity~1.

We also note that the realization  theorem solves a problem raised by J.~van de Leur in \cite{vdL}.
In that paper,  the author showed using a case-by-case argument
that the centreless core of any bi-affine Lie algebras can be
constructed as a multiloop algebra.  (Bi-affine Lie algebras are extended affine
Lie algebras of nullity 2 whose centreless cores had been previously constructed as loop algebras
of affine algebras.)  He pointed out that this was an experimental fact that
was not explained by a general theory.   Since centreless cores of bi-affine Lie algebras
are fgc \cite[Thm.~5.5(iii)]{ABP}, our theorem provides such an explanation  .

To complete our understanding of the realization of EALAs, it is important to
know when two multiloop Lie tori
give rise to the same family of EALAs.  To this end,
we define an \emph{isotope} of a Lie $\Lm$-torus $\cL$ as a Lie $\Lm$-torus
obtained from $\cL$ by shifting the $\Lm$-grading by a group homomorphism
in $\Hom(Q,\Lm)$.  Two Lie tori $\cL$ and $\cL'$, graded by $Q\times \Lm$
and $Q'\times \Lm'$ respectively, are said to \emph{bi-isomorphic} if there is a Lie
algebra isomorphism from $\cL$ to $\cL'$ that preserves the gradings up
to an isomorphism of $Q$ with $Q'$ and and isomorphism of $\Lm$ with $\Lm'$. Further, they are said to be
\emph{isotopic} if one is bi-isomorphic to an isotope of the other.
It turns out that isotopy is exactly what is needed
to distinguish between families of EALAs.
Indeed, it is shown in \cite[Thm.~6.2]{AF} that if two centreless Lie tori
$\cL$ and $\cL'$ are isotopic, then the corresponding families of EALAs are isomorphic,
in the sense that there is a bijection from the parameter set for the first
family onto the parameter set for the second family such that
corresponding EALAs are isomorphic.  Conversely, if some member of the first family is isomorphic as an EALA
to some member of the second family, then $\cL$ is isotopic to $\cL'$ \cite[Thm.~6.1]{AF}.
Hence families of EALAs up to isomorphism
are in one-to-one correspondence with isotopy classes of centreless Lie tori.

The next two theorems of the paper, Theorems \ref{thm:loopbi}
and \ref{thm:isotopy}, give necessary and sufficient conditions for two
multiloop Lie $\Zn$-tori $\LT(\fs,\bsg,\fh)$ and  $\LT(\fs',\bsg',\fh')$ to be
bi-isomorphic and isotopic respectively. In view of the above discussion,
the second theorem (which is proved using
the first theorem)
provides necessary
and sufficient conditions for the families of EALAs corresponding
to $\LT(\fs,\bsg,\fh)$ and  $\LT(\fs',\bsg',\fh')$ to be
isomorphic.  The conditions in both theorems are
conditions on the  input data  $(\fs,\bsg,\fh)$ and  $(\fs',\bsg',\fh')$.  In particular, the
conditions for bi-isomorphism are the existence of an isomorphism
$\ph: \fs \to \fs'$ and a matrix $P\in \GL_n(\bbZ)$ so that
$\bsg'= \ph \bsg^P\ph^{-1}$, where $(\bsg,P) \to \bsg^P$ is the natural
right action of $\GL_n(\bbZ)$ on the set of length $n$ sequences of commuting finite order
automorphisms of $\fs$.

The primary tools in the proof of the theorems just described are the general results in
\cite{ABFP} on realizations of graded-simple algebras.

In the last section of the paper, we consider the classification problems, up to bi-isomorphism and up to isotopy,
for fgc centreless Lie tori of nullity $n\ge 1$.  (As we've mentioned,
the last of these problems is equivalent to the classification
of families of EALAs of nullity $n$ up to isomorphism.)
These problems can be approached in two ways.
The first approach uses the coordinatization theorems which have been proved
for centreless Lie tori of each type $\De$.  (These theorems are the work of many
authors.  See \cite{AF} for a survey of the literature and a sample
of these results.)  Both classification problems have been solved for some types
in this way \cite{AF}, but other types remain to be considered.

In the last section of this article, we
propose a second approach, using multiloop Lie tori and our results, to the classification
problems.  In this approach we organize
the fgc centreless Lie tori by absolute type (see subsection \ref{subsec:absolute})
rather than by type.
Carrying the approach through for a given absolute type X$_\ell$ requires
precise information about sequences of commuting finite order
automorphisms (or equivalently gradings by a finite abelian group)
of a simple Lie algebra of type X$_\ell$ over~$k$.
The approach to the classification up to bi-isomorphism is summarized in Theorem \ref{thm:class}.
Once the classification up to bi-isomorphism is complete for a given
absolute type, one can, at least in principal, apply Theorem \ref{thm:isotopy}
to determine which of the representatives of the bi-isomorphism classes are isotopic.
However, it would be interesting to have a more precise method to accomplish this last step.

As an example we solve both classification problems for absolute type
F$_4$ using the work of Draper and Martin \cite{DM}
on gradings of the simple Lie algebra of type F$_4$
(see Examples \ref{ex:F4} and \ref{ex:F4isotopy}).
We expect that the approach will also work for (at least some) other types,
using classical and recent results about sequences of commuting automorphisms or gradings
(see for example \cite{BFM, BSZ, GP, KS}).

To close this introduction, we briefly outline the contents of the paper.
In section~\ref{sec:Lietori}, we recall or prove the properties that we will need
about Lie tori.
In section \ref{sec:biandiso}, we introduce bi-isomorphism
and isotopy for Lie tori.  Section
\ref{sec:multrealization} contains the construction of multiloop Lie tori
and the statement and proof of the realization theorem.  In section
\ref{sec:loopbiandiso} we prove our theorems on bi-isomorphism and isotopy of multiloop
Lie tori.  Finally, in section \ref{sec:approach},
we describe our approach to the
classification problems.

\medskip\noindent
\textbf{Acknowledgement:}  The authors thank Christina Draper for a very helpful
conversation related to her work with Candido Martin in \cite{DM} on gradings  of
$F_4$.

\medskip\noindent
\textbf{Assumptions and notation:}
Throughout this work \emph{we assume that $k$ is a field of
characteristic $0$}. All algebras are assumed to be algebras over
$k$.  We also \emph{assume that $\Lm$ is an abelian group} written
additively. (Beginning  in section \ref{sec:multrealization} we will make additional assumptions on
$k$ and $\Lm$, since we will be dealing specifically with
multiloop algebras at that point.)

If $\cB = \oplus_{\lm\in\Lm}$ is a $\Lm$-graded algebra we let
$\supp_\Lm(\cB) = \set{\lm\in\Lm\suchthat \cB^\lm\ne 0}$ denote the $\Lm$-\emph{support} of $\cB$.
If $\cB$ is a $\Lm$-graded algebra and $\cB'$ is a
$\Lm'$-graded algebra, we say that $\cB$ and $\cB'$ are \textit{isograded-isomorphic}
if there exists an algebra isomorphism $\ph: \cB \to \cB'$ and a group
isomorphism $\ph_\mathrm{gr} : \Lm \to \Lm'$ such that
$\varphi(\cB^\lm) = \cB'^{\ph_\mathrm{gr}(\lm)}$ for $\lm\in \Lm$.  If
$\langle\supp_\Lm(\cB)\rangle= \Lm$, then $\ph_\mathrm{gr}$
is uniquely determined by $\ph$.

If $\cL$ is a Lie algebra, we denote the \emph{centre} of $\cL$ by $Z(\cL)$, and we call
$\cL$ \emph{centreless} if $Z(\cL) = 0$.

If $S$ is a subset of a group, we denote the subgroup
generated by $S$ as $\langle S \rangle$.

\section{Lie tori}
\label{sec:Lietori}
\subsection{Definitions and basic properties}
\label{subsec:Lietoribasic}
We begin by recalling the definition of a Lie torus and describing  some of its properties.
In order to discuss Lie tori we first establish some conventions
and notation for root systems.

\begin{convention}  As in \cite{AABGP}, \cite{AG} and \cite{N1}
it will be convenient
for us to work with root systems that contain 0.
So by a \textit{finite irreducible root system} we will mean
a finite subset $\De$ of a finite dimensional vector space
$\cX$  over $k$ such that $0\in \De$ and
$\Dec := \De \setminus\set{0}$ is
a finite irreducible root system in $\cX$ in the usual sense
(see \cite[chap.~VI, \S 1, D\'efinition~1]{B1}).
\end{convention}

\emph{Assume  for the rest of this paper (unless indicated
to the contrary)  that $\De$ is a finite irreducible root system in
a finite dimensional vector space $\cX$ over $k$.}  Thus, $\De$ has
one of the following types:
\begin{equation}
\label{eq:type}
\mathrm{A}_\ell\, (\ell \ge 1), \mathrm{B}_\ell\,  (\ell \ge 2),
\mathrm{C}_\ell\,  (\ell \ge 3), \mathrm{D}_\ell\,  (\ell \ge 4), \mathrm{E}_6, \mathrm{E}_7,
\mathrm{E}_8, \mathrm{F}_4, \mathrm{G}_2, \mathrm{BC}_\ell\,  (\ell \ge 1)
\end{equation}
\cite[chap.~VI, \S 4]{B1}.  It will sometimes be convenient to say that $\De$
has type B$_\ell$ with $\ell \ge 1$, which means that $\De$ has type
A$_1$ or B$_\ell$  ($\ell \ge 2)$. Recall that $\De$ is  said to be
\textit{reduced} if $2 \al \notin \Dec$ for $\al \in \Dec$.  Then,
$\De$ is not reduced if and only if $\De$ is of type BC$_\ell$,
where  $\ell \ge 1$.

\begin{notation}
\label{not:rootsystem}
\emph{We will use  the following notation  for the root system $\De$.}
Let
\[Q = Q(\De) := \spann_\bbZ(\De)\]
be the \emph{root lattice} of $\De$.
Let $\cX^*$ denote the dual space of $\cX$,
and let $\langle\ ,\ \rangle : \cX\times \cX^* \to k$
denote the natural pairing of
$\cX$ with $\cX^*$.
If $\al\in \Dec$, $\al^\vee$
will denote the \textit{coroot} of $\al$ in $\cX^*$; that is,
$\al^\vee$ is the unique element of $\cX^*$ such that
$\langle \al,\al^\vee\rangle = 2$ and the map
$\beta \to \beta - \langle \beta,\al^\vee\rangle \al$
stabilizes $\De$.
Let
\[\Desh = \text{set of short roots in } \De;\]
that is $\Desh$ is the set of nonzero roots of minimum length in $\De$.
Let
\[\Deic := \set{\al\in \Dec \suchthat
\textstyle \frac 12 \al\notin \De}\]
denote the set of \textit{indivisible}
roots in $\De$, and let
\[\Dei := \Deic\cup\set{0}.\]
Going in the reverse direction, we can \emph{enlarge} $\De$ by setting
\[\Deen = \left\{
\begin{array}{ll}
    \De\cup2\Desh, & \hbox{if $\De$ has type B$_\ell$, $\ell \ge 1$;} \\
    \De, & \hbox{otherwise.} \\
\end{array}%
\right.\]
Note that both $\Dei$ and $\Deen$ are irreducible root systems
in $\cX$,  and $\Dei$
is reduced.
\end{notation}

\begin{remark}
\label{rem:isocoroot}
If $\De$ is a finite irreducible root system in $\cX$
and $\al\to \tilde\al$ is a $k$-linear isomorphism of $\cX$ onto
a vector space $\tilde\cX$, then
$\widetilde{\De} := \set{\tilde\al \suchthat \al\in \De}$ is a finite
irreducible root system in $\tilde\cX$ and we have
\begin{equation}
\label{eq:isoroot} \langle \al , \beta^\vee\rangle = \langle
\tilde\al , \tilde\beta^\vee\rangle
\end{equation}
for $\al\in \De$ and $\beta\in\Dec$,
where $\tilde\beta^\vee$ is the coroot of $\tilde\beta$
in~$\tilde\cX^*$.
\end{remark}

\begin{example}
\label{def:rootspace}
Suppose  that $\fh$ is a finite dimensional ad-diagonalizable subalgebra of
a Lie algebra $\fg$.  Let
\[\fg_\al = \set{ x \in \fg \suchthat [h,x] = \al(h)x \text{ for } h\in \fh}.\]
Then,
we have the decomposition $\fg = \oplus_{\al\in \fh^*}\fg_\al$, called the \emph{root space
decomposition} of $\fg$ relative to the adjoint action of $\fh$.  We set
\[\De(\fg,\fh) = \set{\al\in \fh^* \suchthat \fg_\al \ne 0}.\]
In many cases, although not in general, $\De(\fg,\fh)$ is a finite  irreducible
root system in $\fh^*$.
In particular, if $\fg$ is finite dimensional split simple
and $\fh$ is a splitting Cartan subalgebra of $\fg$,
then $\De(\fg,\fh)$ is a reduced finite irreducible  root system in $\fh^*$
\cite[\S IV.1]{J}, \cite[chap.~VIII, \S 2, th.~2]{B2}.
\end{example}

Lie tori are Lie algebras that are graded by the direct product
$Q\times \Lm$ and satisfy some additional  axioms that we are going
to recall.  To do this we first
introduce some  notation
for $Q\times \Lm$-graded  algebras.

\begin{notation}
\label{not:doublegraded}
Let
\[\cL = \oplus_{(\al,\lm)\in Q\times\Lm}\dL{\al}{\lm}\]
be a $Q\times\Lm$-graded algebra.   (It is convenient
as in \cite{AG}, \cite{Y2} and \cite{N1} to use the notation
$\dL{\al}{\lm}$ rather than $\cL^{(\al,\lm)}$ or $\cL_{(\al,\lm)}$
for the space of elements of degree $(\al,\lm)$ in $\cL$.)
Then
$\cL = \oplus_{\lm \in \Lm} \cL^\lm$ is
$\Lm$-graded and $\cL = \oplus_{\al \in Q} \cL_\al$ is $Q$-graded algebra,
with
\[\cL^\lm = \oplus_{\al\in Q} \dL{\al}{\lm} \text{ for } \lm\in \Lm
\quad \andd \quad \cL_\al = \oplus_{\lm\in \Lm} \dL{\al}{\lm} \text{
for }\al\in Q,\]
and we have
\[\dL{\al}{\lm} = \cL_\al \cap \cL^\lm.\]
Conversely if $\cL$ has a $Q$-grading and a $\Lm$-grading
that are compatible (which means that each $\cL_\al$ is a $\Lm$-graded
subspace of $\cL$ or equivalently that each $\cL^\lm$ is a $Q$-graded subspace of $\cL$),
then $\cL$ is $Q\times \Lm$-graded with
$\dL{\al}{\lm} = \cL_\al \cap \cL^\lm$.
Hence, a $Q\times \Lm$-graded algebra $\cL$
has  three different associated support sets, namely the $Q\times \Lm$-support,
the $Q$-support and the $\Lm$-support denoted respectively by
$\supp_{Q\times \Lm}(\cL)$, $\supp_{Q}(\cL)$ and $\supp_{\Lm}(\cL)$.
\end{notation}

\medskip
We now recall the definition of a Lie torus.
Lie tori were introduced by Y.~Yoshii   \cite{Y2,Y3} and further studied
by E.~Neher in \cite{N1,N2}.
(See Remark \ref{rem:historyLietorus} below.)

\begin{definition}
\label{def:Lietorus}
A \textit{Lie $\Lm$-torus of type $\Delta$} is a $Q\times\Lm$-graded
Lie algebra $\cL$
over $k$ which (with the notation of \eqref{not:doublegraded}) satisfies:
\begin{itemize}
\item[(LT1)]  $\supp_Q(\cL) \subseteq \De$.

\item[(LT2)]
\begin{itemize}
\item[(i)] $(\Deic,0)\subseteq \supp_{Q\times \Lm}(\cL)$.
\item[(ii)]
If $(\al,\lm)\in \supp_{Q\times \Lm}(\cL)$ and $\al\ne 0$,
then there exist elements $e_\al^\lm\in \dL{\al}{\lm}$
and $f_\al^\lm\in \dL{-\al}{-\lm}$ such that
\[\dL{\al}{\lm} = ke_\al^\lm,\quad \dL{-\al}{-\lm} = kf_\al^\lm,\]
and
\begin{equation}
\label{eq:basic}
[[e_\al^\lm,f_\al^\lm],x_\beta] = \langle\beta,\al^\vee\rangle x_\beta
\end{equation}
for $x_\beta\in \cL_\beta$,  $\beta\in Q$.
\end{itemize}
\item[(LT3)] $\cL$ is generated as an algebra by the spaces $\cL_\al$, $\al\in \Dec$.
\item[(LT4)] $\langle\supp_\Lm(\cL)\rangle =\Lm$.
\end{itemize}
In  that case we call $Q$ the \emph{root grading group} and we call
the $Q$-grading of $\cL$  the \emph{root grading}.
Similarly, we call $\Lm$ the \emph{external grading group}  and we call
the $\Lm$-grading of $\cL$  the \emph{external grading}.
If $\De$ has type X$_\ell$,
with X$_\ell$ as in \eqref{eq:type}, we also say that \emph{$\cL$ has type
X$_\ell$}. Finally if $\Lm$ is free abelian of rank $n\ge 0$,
we say that $\cL$ has \emph{nullity} $n$.
\end{definition}

\begin{remark}
\label{rem:sl2}
Suppose that $\cL$ is a Lie $\Lm$-torus of type $\Delta$.  Unless indicated otherwise,
we  will assume that we have made a fixed choice of elements $e_\al^\lm$ and $f_\al^\lm$ as in
(LT2)(ii) for $(\al,\lm)\in \supp_{Q\times \Lm}(\cL)$ with $\al\in \Dec$.
Note that both the elements
$e_\al^\lm$ and $f_\al^\lm$ are nonzero,
since otherwise we would have $\langle\beta,\al^\vee\rangle=0$
for all $\beta\in \supp_Q(\cL)$ and hence for all $\beta\in\cX$.
Consequently the triple
$(e_\al^\lm, [e_\al^\lm,f_\al^\lm],f_\al^\lm)$
forms an $\sltwo$-triple
in~$\cL$.
\end{remark}

\begin{remark}
\label{rem:LT3}
If $\cL$  is a $Q\times \Lm$-graded algebra, then the
subalgebra of $\cL$ generated by the spaces
$\cup_{\al\in\Dec} \cL_\al$ is a
$Q\times \Lm$ graded ideal of $\cL$.  Moreover,
if $\cL$ satisfies (LT2)(i) this ideal is not zero.
Thus, if $\cL$ is a graded-simple $Q\times \Lm$-graded algebra
satisfying (LT2)(i), then (LT3) automatically holds.
\end{remark}

\begin{remark}
\label{rem:Weylaut}
If  $\cL$ is a Lie $\Lm$-torus of type $\Delta$,  $(\al,\lm)\in\supp_{Q\times\Lm}(\cL)$ and   $\al \ne 0$,
we define $\theta_\al^\lm\in \Aut_k(\cL)$  by
\[\theta_\al^\lm = \exp(\ad(e_\al^\lm))\exp(\ad(-f_\al^\lm)) \exp(\ad(e_\al^\lm)).\]
Then, as noted in \cite[\S 5]{Y3} (see the argument in \cite[Prop.~1.27]{AABGP}), we have
\begin{equation}
\label{eq:Weylaut}
\theta_\al^\lm(\cL_\beta^\mu) = \cL_{w_\al(\beta)}^{\mu-\langle \beta,\al^\vee\rangle \lm}
\end{equation}
for $(\beta,\mu)\in Q\times \Lm$, where $w_\al$ is the Weyl reflection corresponding
to  $\al\in \De$.
\end{remark}

We next examine the $Q$-support of a Lie torus.

\begin{lemma}
\label{lem:suppQ}
Let  $\cL$ be a Lie $\Lm$-torus of type $\De$.
Then
\[ \supp_Q(\cL) = \left\{
\begin{array}{ll}
     \De & \hbox{if $\De$ is reduced;} \\
     \De \text{ or } \Dei & \hbox{if $\De$ is not reduced.} \\
\end{array}
\right.\] Consequently,  $\supp_Q(\cL)$ is a finite irreducible root system in $\cX$
and
\begin{equation*}
\Dei = \supp_Q(\cL)_{\mathrm{ind}}.
\end{equation*}
\end{lemma}

\begin{proof} By (LT1) and (LT2)(i), we have $\Deic  \subseteq \supp_Q(\cL) \subseteq \De$.
Moreover $0\in \supp_Q(\cL)$ since $0\ne [e_\al^0,f_\al^0] \in \cL_0$ for
$\al\in \Deic$.  Thus
\begin{equation}
\label{eq:supptrapped}
\Dei  \subseteq \supp_Q(\cL) \subseteq \De.
\end{equation}
If $\De$ is reduced then $\De = \Dei$ which implies our conclusion.  Suppose that $\De$ is not reduced.
Now $\supp_Q(\cL)$ is closed under the action of the
Weyl group of $\De$ by  \eqref{eq:Weylaut}.
This fact together with \eqref{eq:supptrapped} implies our conclusion.
\end{proof}

\begin{remark}
\label{rem:LT5}
Suppose that $\cL$ is a Lie $\Lm$-torus of type $\De$.
It is sometimes convenient to assume the following additional axiom:
\begin{itemize}
\item[(LT5)]  $\supp_Q(\cL) = \De$.
\end{itemize}
Note that if $\cL$ does not satisfy (LT5), then by
Lemma \ref{lem:suppQ} it follows that $\De$ is not reduced
and $\supp_Q(\cL) = \Dei$.  In that case $\cL$
is also a Lie $\Lm$-torus of type $\Dei$ and as such
it satisfies (LT5).  Thus there is really no loss of generality
in assuming (LT5)  when convenient.
\end{remark}

If $\cL$ is a Lie $\Lm$-torus of type $\De$, we define
\begin{equation}
\label{eq:semilattice}
\sem_\al = \set{\lm\in\Lm \suchthat \dL{\al}{\lm} \ne 0}
\end{equation}
for $\al\in Q$.  Thus,
\[\supp_\Lm(\cL) = \textstyle\bigcup_{\al\in Q} \sem_\al,\]
and $\sem_\al \ne \emptyset$ if and only if $\al\in \supp_Q(\cL)$.
Note that if $\al,\beta\in \supp_Q(\cL)$ and $\alpha \ne 0$,
we have, by \eqref{eq:Weylaut},
\begin{equation}
\label{eq:semilatticeprop}
\sem_\beta - \langle \beta,\al^\vee\rangle \sem_\al \subseteq \sem_{w_\al(\beta)}.
\end{equation}

Using \cite{Y3}, we have the following  properties of the sets $\sem_\al$.  (See
also \cite[Chapter II]{AABGP}.)

\begin{lemma}\label{lem:suppLie}
Suppose that $\cL$ is a Lie $\Lm$-torus of type $\Delta$.
\begin{itemize}
\item[(i)] If $\al\in \Deic$, then $0\in \sem_\al$.
\item[(ii)] If $\al\in \supp_Q(\cL)^\times$, then
$\sem_\al + 2\sem_\al \subseteq \sem_\al$ and $-\sem_\al = \sem_\al$.
\item[(iii)] If $\al$, $\beta \in \supp_Q(\cL)^\times$ have the same length, then
$\sem_\al = \sem_\beta$.
\item[(iv)] If $\al\in\supp_Q(\cL)^\times$ and
$\beta\in \Desh$, then $\sem_\al\subseteq \sem_\beta$.
\item[(v)] If $\beta\in \Desh$, then
$\supp_\Lm(\cL)= \sem_0 = \sem_\beta + \sem_\beta$,   $\Lm = \langle \sem_\beta\rangle$,
$\sem_\beta + 2\Lm \subseteq \sem_\beta$ and $2\Lm \subseteq \sem_\beta$.
\item[(vi)] $2\Lm \subseteq \supp_Q(\cL)$.
\end{itemize}
\end{lemma}

\begin{proof} (i) follows from (LT2)(i).  Because
\eqref{eq:semilatticeprop} holds, we may invoke the results of \cite[\S 3]{Y3},
which gives (ii), (iii) and (iv).  For (v), the equality
$\sem_0 = \sem_\beta + \sem_\beta$ is proved in \cite[Thm.~5.1]{Y3}.
Then, since $\supp_\Lm(\cL) = \bigcup_{\al\in Q} \sem_\al$, we see using (iv) and (i) that
$\supp_\Lm(\cL) = \sem_0$.
Hence, since  $\Lm = \langle \supp_\Lm(\cL) \rangle$, we have $\Lm = \langle \sem_\beta\rangle$.
Thus, by (ii) (applied to $\al = \beta$), we have $\sem_\beta + 2\Lm \subseteq \sem_\beta$.
Since $0\in \sem_\beta$, this gives $2\Lm \subseteq \sem_\beta$.
Finally, for (vi), choose $\beta\in \Desh$. Then,
$2\Lm \subseteq \sem_\beta \subseteq \supp_\Lm(\cL)$.
\end{proof}

\subsection{The root grading pair}
\label{subsec:gradingpair}

\begin{definition}
\label{def:gradingpair}
Let $\cL$ be a Lie $\Lm$-torus of type $\Delta$.
Following \cite{N1}, we define the
\textit{root grading pair} for $\cL$ to be the pair $(\fg,\fh)$
of subalgebras of $\cL$ where $\fg$ is
the subalgebra of $\cL$ generated by $\set{\dL{\al}{0}}_{\al\in\Dec}$
and $\fh  = \sum_{\al\in\Dec} [\dL{\al}{0},\dL{-\al}{0}]$.
Notice that $\fh$ is a subalgebra of $\fg$,
$\fg \subseteq \cL^0$ and $\fh \subseteq \dL{0}{0}$.
We will see below in
\eqref{eq:Lroot} that
the subalgebra $\fh$ completely determines the   root grading
of $\cL$.
\end{definition}

The
next proposition summarizes the basic properties of
the root grading pair of a Lie torus.  Parts (i)--(v) were announced
in \cite[\S 3]{N1} in the case when $\Lm$  is a finitely generated free abelian group.
We sketch a proof of the proposition for the
convenience of the reader.  If $\cL$ is any Lie algebra,
we denote the \textit{centre} of $\cL$ by $Z(\cL)$.

\begin{proposition}
\label{prop:Lietorusbasic}
Let  $\cL$ be a Lie $\Lm$-torus of type $\Delta$ with root grading
pair~$(\fg,\fh)$.
\begin{itemize}
\item[(i)] If $\al\in\Deic$, then $\dL{2\al}{0} = 0$.
\item[(ii)]  $\fg$ is a finite dimensional split simple
Lie algebra with splitting Cartan subalgebra~$\fh$.
\item[(iii)]  There is a unique linear isomorphism
$\al \to \tilde\al$ of $\cX$ onto $\fh^*$
such that  $\widetilde{\Dei} = \De(\fg,\fh)$
and
\begin{equation}
\label{eq:efvee0}
[e_\al^0,f_\al^0] = \tilde\al^\vee
\end{equation}
for $\al\in\Deic$.
(See Remark \ref{rem:isocoroot} for the notation.  Here $\tilde\al^\vee\in (\fh^*)^* = \fh$.)
\item[(iv)] If $\al\in Q$ then
\begin{equation}
\label{eq:Lroot}
\cL_\al = \set{x\in \cL \suchthat [h,x] = \tilde\al(h) x \text{ for } h\in \fh}.
\end{equation}
\item[(v)] If $(\al,\lm)\in \supp_{Q\times \Lm}(\cL)$ and $\al\in\Dec$ then
\begin{equation}
\label{eq:efvee}
[e_\al^\lm,f_\al^\lm] - \tilde\al^\vee \in Z(\cL).
\end{equation}
\item[(vi)] We have
\[\cL^0 = \fg \oplus \left(Z(\cL)\cap \cL^0\right) \andd
\dL{0}{0} = \fh \oplus \left(Z(\cL)\cap \cL^0\right),\]
so
\[\fg = [\cL^0, \cL^0]
\andd
\fh = [\cL^0, \cL^0] \cap \cL_0.\]
Consequently, if $\cL$  is centreless, then $\cL^0 = \fg$ and $\cL_0^0 = \fh$.
\end{itemize}
\end{proposition}

\begin{proof}
(i):  This is easily proved by contradiction
using the representation theory
of the $\sltwo$-triple
$(e_\al^0, [e_\al^0,f_\al^0],f_\al^0)$
and the fact that
$[\dL{\al}{0},\dL{-\al}{0}] = k[e_\al^0,f_\al^0]$.  (See for example,
the proof of Theorem 1.29(c) of \cite[Prop.~6.3]{AABGP}.)

(ii)--(v): Uniqueness in (iii) is clear, so we need to prove
the other properties in (ii)--(v).
Let $\Pi =\set{\al_1,\dots,\al_\ell}$ be a base for the root system
$\Dei$.
Let
\[e_i = e^0_{\al_i}, \quad f_i = f^0_{\al_i} \andd h_i = [e_i,f_i]\]
for $1\le i\le \ell$.  Let $\tilde\fg$ be the subalgebra of $\fg$ generated
by the elements
$\set{e_i,f_i}_{i=1}^\ell$,
and let $\tilde\fh$ be the subspace
of $\fh$ spanned by $\set{h_i}_{i=1}^\ell$.
(We will see below that $\tilde\fg = \fg$ and $\tilde\fh = \fh$.)
It is easy to check using \eqref{eq:basic}
that the elements $\set{e_i,f_i,h_i}_{i=1}^\ell$ satisfy the
Chevalley-Serre
relations determined by the Cartan matrix
$(\langle \al_i, \al_j^\vee\rangle)$.  Hence,
by  \cite[chap.~VIII, \S 4, th.~1]{B2},
$\tilde\fg$
is a finite dimensional split simple Lie algebra
with splitting Cartan subalgebra $\tilde\fh$,
and there is a linear isomorphism
$\al\to \tilde\al$ such that
$\widetilde{\Dei} = \De(\tilde\fg,\tilde\fh)$  and
\begin{equation}
\label{eq:veehi}
\tilde\al(h_i) = \langle \al,\al_i^\vee\rangle
\end{equation}
for $i=1,\dots,\ell$.

Now we set
$\cL_{\tilde\al} = \set{x\in \cL \suchthat [h,x] = \tilde\al(h)x \text{ for } h\in\tilde\fh}$
for $\al\in Q$.  Then using \eqref{eq:basic} and \eqref{eq:veehi}
 one checks that $\cL_\al \subseteq \cL_{\tilde\al}$
for $\al\in Q$.  Hence since $\cL = \oplus_{\al\in Q} \cL_\al$ and
since the sum $\oplus_{\al\in Q} \cL_{\tilde\al}$ is direct we have
$\cL_\al = \cL_{\tilde\al}$ for $\al\in Q$. So
\begin{equation}
\label{eq:Ltilderoot}
\cL_\al = \set{x\in \cL \suchthat [h,x] = \tilde\al(h)x \text{ for } h\in\tilde\fh}
\end{equation}
for $\al\in Q$.

Next if $\al\in \Deic$, then $\tilde\al$ is a nonzero root of $\tilde\fg$ relative to
$\tilde\fh$, so, by \eqref{eq:Ltilderoot}, we have
$\tilde\fg\cap\cL_\al \ne 0$. Hence
$\tilde\fg\cap\dL{\al}{0} \ne 0$, so, by the 1-dimensionality of $\dL{\al}{0}$,
we have
$\dL{\al}{0}\subseteq \tilde\fg$ for $\al\in \Deic$.  But if
$\al\in \Dec\setminus\Deic$ we have $\dL{\al}{0} = 0$ by (i).  Hence
$\dL{\al}{0}\subseteq \tilde\fg$ for $\al\in \Dec$, so $\fg \subseteq \tilde\fg$.
But by definition  we have $\tilde\fg\subseteq\fg$, so $\tilde\fg = \fg$.
Consequently, $\tilde\fh \subseteq\fh \subseteq \tilde\fg$
and, by \eqref{eq:Ltilderoot}, $\fh$ centralizes $\tilde\fh$.  So $\tilde\fh
= \fh$.

It remains to show \eqref{eq:efvee0} and \eqref{eq:efvee}. Since
$\fg\cap Z(\cL) = 0$, it suffices to show \eqref{eq:efvee}.  For
this  let $(\al,\lm)\in \supp_{Q\times \Lm}(\cL)$ with $\al\in\Dec$.
Then using \eqref{eq:isoroot}, \eqref{eq:basic} and
\eqref{eq:Ltilderoot}, one checks that $[[e_\al^\lm,f_\al^\lm] -
\tilde\al^\vee,x_\beta^\mu] = 0$ for $\beta\in\De$, $\mu\in \Lm$ and
$x_\beta^\mu\in \dL{\beta}{\mu}$. This proves~\eqref{eq:efvee}.

(vi): It
suffices to prove that
$\cL^0 = \fg + Z(\cL)\cap \cL^0$ and
$\dL{0}{0} = \fh + Z(\cL)\cap \cL^0$.
In  fact it's enough to prove the first of these
equations since
$\dL{0}{0}$ is the centralizer
in $\cL^0$ of $\fh$ (by \eqref{eq:Lroot}).
To do this it's enough to show that $\dL{\al}{0} \subseteq \fg + Z(\cL)$
for $\al\in \De$.  But this is true by definition of $\fg$ if $\al\ne 0$.
On the other hand, using (LT3) and  \eqref{eq:efvee}, we have
$\dL{0}{0} =
\sum_{\al\in\Dec}\sum_{\lm\in\Lm} [\dL{\al}{\lm},\dL{-\al}{-\lm}]
\subseteq \fh + Z(\cL)$.
\end{proof}

\begin{remark}
\label{rem:historyLietorus}
As noted already,  Lie $\Lm$-tori were
first defined by Yoshii in \cite{Y2,Y3},
whereas the definition given above in
Definition \ref{def:Lietorus}
is due to Neher \cite{N1}.
Yoshii's original definition assumed that $\cL$
is a  \emph{root graded} Lie algebra of type $\De$
as defined in \cite[Chapter 1]{ABG}.
Yoshii and Neher's definitions are equivalent
in view of Proposition \ref{prop:Lietorusbasic}
(see \cite[\S 3]{N1}).
\end{remark}

\begin{remark}
\label{rem:nullity0}
If  $\Lm = 0$ and $\cL$ is a Lie $\Lm$-torus, then, by
Proposition \ref{prop:Lietorusbasic}~(vi), we have
$\cL = \fg \oplus Z(\cL)$ and therefore, by
(LT4) and \eqref{eq:Lroot}, we have $\cL = \fg$.
Thus, $\cL$ is a finite dimensional split simple Lie  algebra.
\end{remark}

\subsection{Fgc algebras}
\label{subsec:fgcalgebras}

A class of Lie tori of  particular interest from the point
of view of EALA theory is the class of fgc centreless Lie tori.
To discuss this class, we first recall some facts about the centroid and the fgc
condition for algebras in general (not necessarily Lie algebras).
(See \cite[\S2.2]{BN} and \cite[\S4]{ABFP} for more
details.)

\begin{definition} Suppose that $\cB$ is an algebra.
The \emph{centroid} of $\cB$ is the subalgebra
$C_k(\cB)$ of $\End_k(\cB)$ consisting of the $k$-linear endomorphisms
of $\cB$ that commute with all left and right multiplication
by elements of $\cB$.  We usually write $C_k(\cB)$ simply as $C(\cB)$.
$\cB$ is said to be \emph{central-simple}
if $\cB$ is simple and $C(\cB) = k1$.
Recall that a finite dimensional simple algebra
is automatically central-simple if $k$ is algebraically closed \cite[Theorem 10.1]{J}.
\end{definition}

\begin{definition}  Suppose that $\cB = \bigoplus_{\lambda\in\Lambda}\cB^\lm$
is a $\Lm$-graded algebra.  $\cB$ is said to be \emph{graded-simple}
if $\cB\cB \ne 0$ and the only graded ideals of $\cB$ are $0$ and
$\cB$. If $\cB$ is graded-simple then $C({\cB}) =
\bigoplus_{\lambda\in\Lambda} C({\cB})^\lm$ is a unital commutative
associative $\Lm$-graded algebra, where
\[C(\cB)^\lm =
\set{c\in C(\cB) \suchthat c\cB^\mu \subseteq \cB^{\lm+\mu} \text{
for } \mu \in \Lm}\] for $\lm\in \Lm$ \cite[Proposition 2.16]{BN}.
Finally
$\cB$ is said to be \emph{graded-central-simple}
if $\cB$ is graded-simple and $C(\cB)^0 = k1$.
\end{definition}

\begin{remark}
\label{rem:gcs}
If $\cB$ is a graded-simple $\Lm$-graded algebra and
$\dim(\cB^\lm) = 1$ for some $\lm\in \Lm$, then $\cB$ is
graded-central-simple \cite[Lemma 4.3.4]{ABFP}.
\end{remark}

\begin{definition}  Suppose that $\cB = \bigoplus_{\lambda\in\Lambda}\cB^\lm$
is a graded-simple $\Lm$-graded algebra. The support $\Gm_\Lm(\cB) =
\set{\gm\in \Lm \suchthat C(\cB)^\gm \ne 0}$ of $C(\cB)$ in $\Lm$ is
a subgroup of $\Lm$ \cite[Proposition 2.16]{BN} which we call the
\textit{central grading group} of~$\cB$.  We often write
$\Gm_\Lm(\cB)$ simply as $\Gm(\cB)$.
\end{definition}

As noted in \cite[Lemmas 4.3.5 and 4.3.8]{ABFP}, we have the
following:

\begin{lemma}
\label{lem:Cdecomp}
Suppose that $\cB$ is a graded-central-simple $\Lm$-graded algebra.
Then $C(\cB)$  has a basis
$\set{c_\gm}_{\gm\in \Gm(\cB)}$ such that $c_\gm\in C(\cB)^\gm$ is a unit
of $C(\cB)$ for $\gm\in \Gm(\cB)$.  Moreover,
if $\Lm$ is finitely generated and free or if $k$ is algebraically closed,
then $C(\cB)$ is isomorphic as a $\Lm$-graded algebra to the group algebra
$k[\Gm(\cB)]$.
\end{lemma}

\begin{definition}
\label{def:fgc}
If $\cB$ is an algebra, then $\cB$ is naturally a left $C(\cB)$-module.
(In fact if $C(\cB)$ is commutative, $\cB$ is an algebra over $C(\cB)$.)
We say that $\cB$ is \textit{fgc} if
$\cB$ is finitely generated as a $C(\cB)$-module.
\end{definition}

The following lemma is part of
\cite[Proposition 4.4.5]{ABFP}.

\begin{lemma}
\label{lem:fgc}  Suppose that $\Lm$
is finitely generated and that $\cB$ is a
graded-central-simple $\Lm$-graded algebra
so that $\ell\Lm \subseteq \supp_\Lm(\cB)$ for some  $\ell\ge 1$.
Then the following are equivalent:
\begin{itemize}
    \item[(a)] $\cB$ is fgc.
    \item[(b)] $\cB$ is a free module of finite rank over $C(\cB)$.
    \item[(c)] $\Lm/\Gm(\cB)$ is
finite and $\dim(\cB^\lm) < \infty$ for  all  $\lm\in \Lm$.
\end{itemize}
\end{lemma}

\subsection{Fgc centreless Lie tori}
\label{subsec:fgcLT}

\begin{proposition}
\label{prop:Lieprop}
Suppose that
$\cL$ is a centreless Lie $\Lm$-torus of type $\De$.
Then
\begin{itemize}
    \item[(i)] $\cL$ is a graded-central-simple $Q\times\Lm$-graded and a
    graded-central-simple $\Lm$-graded algebra.
    \item[(ii)] $\Gm_{Q\times\Lm}(\cL) = \set{0}\times \Gm_\Lm(\cL)$.
\end{itemize}
\end{proposition}

\begin{proof}  Let $(\fg,\fh)$ be the root grading pair for $\cL$.
Yoshii has shown in \cite[Lemma 4.4]{Y3} that
$\cL$ is a graded-simple $\Lm$-graded algebra.
(The argument uses \eqref{eq:Lroot} and \eqref{eq:efvee}.)  Thus
$\cL$ is also a graded-simple $Q\times\Lm$-graded algebra.
Consequently the groups
$\Gm_{Q\times\Lm}(\cL)$ and  $\Gm_\Lm(\cL)$ are defined.
But since the action of $C(\cL)$ on $\cL$ commutes with the adjoint
action of $\fh$, it follows from \eqref{eq:Lroot} that
$C(\cL)\cL_\al \subseteq \cL_\al$ for $\al\in Q$.
Hence $\supp_{Q\times\Lm}(C(\cL))\subseteq \set{0}\times \Lm$,
so we have (ii).
Finally, since $\cL$ contains
a 1-dimensional homogeneous summand,
$\cL$ is a graded-central-simple
$Q\times \Lm$-algebra by Remark \ref{rem:gcs}.
Furthermore, using this fact and (ii), we have
$C(\cL)^0 = \dCL{0}{0} = k1$.  Thus
$\cL$ is a graded-central-simple $\Lm$-graded  algebra.
\end{proof}

For centreless Lie tori, Lemma \ref{lem:fgc}
simplifies as follows:

\begin{proposition}
\label{prop:fgc}
Suppose that $\Lm$ is finitely generated, and let $\cL$ be a centerless
Lie $\Lm$-torus
of type $\De$.  Then the following statements are equivalent:
\begin{itemize}
\item[(a)] $\cL$ is fgc.
\item[(b)] $\cL$ is a free $C(\cL)$-module of finite rank.
\item[(c)] $\Lm/\Gm_\Lm(\cL)$ is finite.
\end{itemize}
\end{proposition}

\begin{proof}
Now $\cL$ is a graded-central-simple algebra by Proposition
\ref{prop:Lieprop}, and
$2\Lm \subseteq \supp_\Lm(\cL)$ by Lemma \ref{lem:suppLie}(vi).
Thus, by Lemma \ref{lem:fgc},
it is enough to show  that
(c) implies $\dim(\cL^\lm) < \infty$ for $\lm\in \Lm$.
A stronger result than this (which bounds the root spaces
of any extended affine Lie algebra) is
announced by Neher in \cite[Proposition 3(b)]{N2}.
For the convenience of the reader we present
the simple argument that is needed here.

Suppose that $\Lm/\Gm_\Lm(\cL)$ is finite.  We show that
\begin{equation}
\label{eq:boundeddim}
\dim(\dL{\al}{\lm}) \le \tfrac 12\left|\Dec\right|s
\andd \dim(\cL^\lm) \le  {\tfrac 12} \left|\Dec\right|s +\left| \Dec\right|
\end{equation}
for
$\al\in  Q$ and $\lm\in \Lm$, where $s = \left|\Lm/\Gm_\Lm(\cL)\right|$.
Since $\dim(\dL{\al}{\lm}) \le 1$ for $\al \ne 0$, the second equation
in \eqref{eq:boundeddim} follows from the first and it is enough to show the first equation
when $\al = 0$.  To do this, let $\Theta$ be a set of representatives of the cosets of
$\Gm_\Lm(\cL)$ in $\Lm$.
Also let $\Dep$ be the set of positive roots in $\De$ relative to some choice
of base for $\De$.
Then for $\lm\in \Lm$, we have
{\allowdisplaybreaks
\begin{align*}
\dL{0}{\lm} &= \sum_{\al\in \Dep} \sum_{\mu\in \Lm} [\dL{\al}{\mu},\cL_{-\al}^{\lm-\mu}]
&&\text{(by (LT3))}\\
&= \sum_{\al\in \Dep}
\sum_{\theta\in \Theta} \sum_{\gm\in \Gm_{\Lm}(\cL)}
[\dL{\al}{\gm+\theta},\cL_{-\al}^{\lm-\gm-\theta}]\\
&= \sum_{\al\in \Dep}
\sum_{\theta\in \Theta} \sum_{\gm\in \Gm_{\Lm}(\cL)}
[\dCL{0}{\gm}\dL{\al}{\theta},\dCL{0}{-\gm}\cL_{-\al}^{\lm-\theta}]
&&\text{(by Lemma \ref{lem:Cdecomp})}\\
&= \sum_{\al\in \Dep}
\sum_{\theta\in \Theta} \sum_{\gm\in \Gm_{\Lm}(\cL)}
\dCL{0}{\gm}\dCL{0}{-\gm}[\dL{\al}{\theta},\cL_{-\al}^{\lm-\theta}]\\
&= \sum_{\al\in \Dep} \sum_{\theta\in \Theta} [\dL{\al}{\theta},\cL_{-\al}^{\lm-\theta}].
\end{align*}
}
Thus we have $\dim(\dL{0}{\lm}) \le
\left|\Dep\right|\left|\Theta\right| = \tfrac 12\left|\Dec\right|s$.
\end{proof}

\begin{remark}
\label{rem:almostallfgc}
Suppose that $k$ is  algebraically closed (of characteristic 0)
and $\Lm$ is finitely generated free abelian group
of rank~$n$. It is a remarkable fact, announced by Neher in
\cite[Theorem~7(b)]{N1}, that if $\Delta$ is not of type A$_{\ell}$ for $\ell \ge 1$, then
\textit{any} centerless Lie $\Lm$-torus of type $\De$ is fgc.
Furthermore,  one can see from the classification
of centerless Lie $\Lm$-tori of type A$_\ell$ \cite{BGK,BGKN,Y1} that there is only one
family of centerless Lie $\Lm$-tori of type A$_\ell$ that are not
fgc, namely the  family of Lie algebras
of the form $\operatorname{sl}_{\ell+1}(k_{\mathbf q})$,
where   $k_{\mathbf q}$ is the quantum torus associated with
an $n\times n$-quantum matrix $\mathbf q$ containing an entry
that is not a root of unity.  Thus in some sense ``almost all''
centerless Lie tori are  fgc.
\end{remark}

\section{Bi-isomorphism and isotopy  for Lie tori}
\label{sec:biandiso}

In  this section we introduce and investigate
two important notions of isomorphism for Lie tori:
bi-isograded-isomorphism, or bi-isomorphism for short, and isotopy.
(See also \cite{AF} for more information about this topic.)

\subsection{Bi-isomorphism}
\label{subsec:biLT}

\begin{definition}
\label{def:biLT}
If $\cL$ is a $Q\times \Lm$-graded algebra
and $\cL'$ is a $Q'\times \Lm'$-graded algebra, we say that
$\cL$ and $\cL'$ are \textit{bi-isograded-isomorphic}
if there is an algebra isomorphism from $\cL$ to $\cL'$ that is
isograded relative to both the root grading and the external grading;
this means that there is an algebra isomorphism $\ph : \cL \to \cL'$,
a group isomorphism $\ph_r: Q \to Q'$,
and a group isomorphism  $\ph_e : \Lm \to \Lm'$ such that
\[\ph(\dL{\al}{\lm}) = \dLp{\ph_r(\al)}{\ph_e(\lm)}\]
for $\al\in Q$ and $\lm\in \Lm$.  In that case
we call $\ph$ a \emph{bi-isograded-isomorphism}.
If $\cL$ and $\cL'$ are Lie tori, then since
$\langle\supp_Q(\cL)\rangle= Q$ and
$\langle\supp_\Lm(\cL)\rangle= \Lm$,
the maps $\ph_r$ and $\ph_e$ are uniquely determined and we call
these maps respectively the \emph{root-grading isomorphism}
and the \emph{external-grading isomorphism} corresponding to $\ph$.
We will usually abbreviate bi-isograded-isomorphic
and bi-isograded-isomorphism as \emph{bi-isomorphic} and
\emph{bi-isomorphism}  respectively.
\end{definition}

\begin{remark}
\label{rem:equivLT1}
Suppose  that   $\cL$ is a $\Lm$-torus of type $\De$
with root grading pair $(\fg,\fh)$,
and $\cL'$ is  $\Lm'$-torus of type~$\De'$ with root grading pair $(\fg,\fh)$.
Suppose that $\ph$ is a bi-isomorphism of $\cL$ onto $\cL'$.
Then it is clear that $\ph(\fg) =\fg'$ and
$\ph(\fh) = \fh'$, so, in this sense, $\cL$ and $\cL'$
have isomorphic root grading pairs.
Also, by the last statement in Lemma \ref{lem:suppQ},
we have
\[\ph_r(\Dei) = \Dei.\]
Thus if $\cL$ and $\cL'$ satisfy (LT5),
we have $\ph_r(\De) = \De'$, so $\De$ and $\De'$ are
isomorphic root   systems.
\end{remark}

We now see that for the purposes of  determining whether
or not two Lie tori are bi-isomorphic, we can
ignore the root gradings.

\begin{proposition}
\label{prop:torusequiv}
Suppose that $\cL$ is a Lie $\Lm$-torus of type $\De$ and $\cL'$ is
a Lie $\Lm'$-torus of type $\De'$.
Then $\cL$ and $\cL'$ are bi-isomorphic
if and only if there is an algebra isomorphism
$\ph : \cL \to \cL'$ and a group isomorphism $\ph_e : \Lm \to \Lm'$
such that
$\ph(\cL^\lm) = \cL'^{\,\ph_e(\lm)}$
for $\lm\in \Lm$.
\end{proposition}

\begin{proof}  The implication
``$\Rightarrow$'' is immediate from the definition.

Conversely, suppose that we have an
isomorphism $\ph: \cL \to \cL'$ and a group isomorphism $\ph_e :
\Lm \to \Lm'$ such that
\[\varphi(\cL^\lm) = \cL'^{\,\ph_e(\lm)}\]
for
$\lm\in \Lm$.  Let $Q$ be the root lattice for $\cL$
and let $(\fg,\fh)$ be the root grading pair for $\cL$.
Let $Q'$ and $(\fg',\fh')$ be the corresponding
objects for $\cL'$.
By Proposition \ref{prop:Lietorusbasic}(iii) and (iv),
we may identify $\De$ (via the map $\al \mapsto \tilde\al$)
with a root system in $\fh^*$ such that
$\Dei$ is the set of roots (including 0) of $\fg$ relative to
$\fh$ and
\begin{equation}
\label{eq:Lroot1}
\cL_\al = \set{x\in \cL \suchthat [h,x] = \al(h) x \text{ for } h\in \fh}
\end{equation}
for $\al\in Q$.
Similarly
we may identify $\De'$
with a root system in ${\fh'}^*$ such that
$\Deip$ is the set of roots (including 0) of $\fg'$ relative to
$\fh'$ and
\begin{equation}
\label{eq:Lroot2}
\cL'_{\al'} = \set{x'\in \cL' \suchthat [h',x'] = \al'(h') x' \text{ for } h'\in \fh'}
\end{equation}
for $\al'\in Q'$.
We fix a base for the root system $\De$,
so we have a notion of positive and negative roots in $\De$.
Let $\fn_+$ (resp.~$\fn_-$) denote the
sum of the root spaces of $\fg$ relative to $\fh$
corresponding to positive (resp.~negative)
roots.  So $\fg$ has the triangular decomposition
$\fg = \fn_+ \oplus \fh  \oplus \fn_-$.

Observe now that
by  Proposition \ref{prop:Lietorusbasic}(vi) we have
\[\ph(\fg) = \fg'.\]
So both $\fh$ and $\ph^{-1}(\fh')$ are splitting Cartan subalgebras
for $\fg$.  Hence,  there exists
an automorphism $\psi_\fg$ of $\fg$
such that $\psi_\fg(\ph^{-1}(\fh')) = \fh$.  In
fact, by  \cite[Corollary on p.~28]{Sel}, $\psi_\fg$ can be chosen of the form
\[\psi_\fg = \exp(\ad_\fg(x_1)) \dots \exp(\ad_\fg(x_\ell)),\]
where  $x_i \in \fn_-\cup \fn_+$ for $1\le i\le \ell$. Now since
$x_i \in \fn_-\cup \fn_+$, it follows that $\ad_\cL(x_i)$ is
nilpotent for $1\le i\le \ell$. So we may define
\[\psi_\cL = \exp(\ad_\cL(x_1)) \dots \exp(\ad_\cL(x_\ell))\in \Aut_k(\cL).\]
Then $\psi_\cL$ extends $\psi_\fg$,
so
\[\psi_\cL(\ph^{-1}(\fh'))= \fh.\]
Also since $x_i\in \fg \subseteq \cL^0$ for each $i$, we have
$\psi_\cL(\cL^\lm) \subseteq \cL^\lm$ for $\lm\in \Lm$.
Therefore since $\psi_\cL\in \Aut_k(\cL)$ we have
$\psi_\cL(\cL^\lm) = \cL^\lm$
for $\lm\in \Lm$.  Thus replacing
$\ph$ by $\ph \circ (\psi_\cL)^{-1}$ we may assume
that
\begin{equation}
\label{eq:stabfh}
\ph(\fh) = \fh'.
\end{equation}

If $\al\in \fh^*$, we define $\cL_\al$ by the equality \eqref{eq:Lroot1}.
Similarly, if $\al'\in {\fh'}^*$ we define $\cL'_{\al'}$ by the equality \eqref{eq:Lroot2}.
Also, let $\hat\ph : \fh^*\to {\fh'}^*$ denote the inverse dual of the linear map
$\ph|_{\fh} : \fh \to \fh'$;  that is $\hat\ph$ is defined by
$\hat\ph(\al)(h') = \al(\ph^{-1}(h'))$
for $\al\in \fh^*$ and $h'\in \fh'$.  Then it follows from \eqref{eq:stabfh}
that
\[\ph(\cL_\al) = \cL'_{\hat\ph(\al)}\]
for $\al\in \fh^*$.  Consequently, we have
$\hat\ph(\supp_Q(\cL)) = \supp_{Q'}(\cL')$.
But by Lemma \ref{lem:suppQ}, we have $\spann_{\bbZ}(\supp_Q(\cL)) = Q$ and
$\spann_{\bbZ}(\supp_{Q'}(\cL')) = Q'$.
Therefore $\hat\ph(Q) = Q'$, and we may define $\ph_r = \hat\ph\mid_{Q} : Q\to Q'$.
Then for $\al\in Q$ and $\lm \in\Lm$, we have
\[\ph(\dL{\al}{\lm}) = \ph(\cL_\al\cap\cL^\lm) = \ph(\cL_\al)\cap\ph(\cL^\lm)
= \cL'_{\ph_r(\al)}\cap{\cL'}^{\,\ph_e(\lm)}
= \dLp{\ph_r(\al)}{\,\ph_e(\lm)}.\qedhere
\]
\end{proof}

\subsection{Isotopes and isotopy}
\label{subsec:isotopes}

Suppose in this  subsection that \emph{$\cL$ is a Lie $\Lm$-torus of type $\De$, and let
$Q$ be the root lattice of $\De$}.

\begin{definition}
\label{def:homotope}
Suppose that $\shom \in \Hom(Q,\Lm)$, where  $\Hom(Q,\Lm)$ is the group
of group homomorphisms from $Q$ into $\Lm$.  We define a new $Q\times \Lm$-graded
Lie algebra $\cLs$ as follows.  As a Lie algebra, $\cLs = \cL$.
The grading on $\cLs$ is given by
\begin{equation}
\label{eq:isograding}
(\cLs)^\lm_\al = \dL{\al}{\lm+\shom(\al)}
\end{equation}
for $\al\in Q$, $\lm\in\Lm$.
\end{definition}

\begin{remark}
\label{rem:homotope}
Let $\shom \in \Hom(Q,\Lm)$.

(i) Observe that $(\cLs)_\al = \cL_\al$ for $\al\in Q$, and hence $\supp_Q(\cLs) = \supp_Q(\cL)$.

(ii) Recall that for $\al\in Q$ we have defined
$\sem_\al = \set{\lm\in\Lm \suchthat \dL{\al}{\lm} \ne 0}$
(see \eqref{eq:semilattice}).
Then, if $\al\in Q$, we have $(\cLs)_\al^0 = \dL{\al}{\shom(\al)}$, so
\begin{equation}
\label{eq:isograding0}
(\cLs)_\al^0 \ne 0 \iff \shom(\al) \in \Lm_\al.
\end{equation}
\end{remark}

\begin{proposition}
\label{prop:isotope}
Let $\shom\in \Hom(Q,\Lm)$, and let $\Pi$ be a base for the root system~$\De$.
The following statements are equivalent:
\begin{itemize}
\item[(a)] $\cLs$ is a Lie torus.
\item[(b)] $\shom(\al) \in \sem_\al$ for all $\al\in \Deic$.
\item[(c)] $\shom(\al) \in \sem_\al$ for all $\al\in \Pi$.
\end{itemize}
Moreover, if these conditions are satisfied, then $\cLs$ satisfies
(LT5) if and only if $\cL$ satisfies (LT5).
\end{proposition}

\begin{proof} The final statement is clear using
Remark \ref{rem:homotope}(i).  So we need to prove the equivalence of (a), (b) and (c).
To simplify notation, let  $\tL = \cLs$.

``(a)$\Leftrightarrow$(b)'' If (a) holds then $\tL$ satisfies (LT2)(i), and hence,
by \eqref{eq:isograding0}, (b) holds.  For the converse, suppose that (b) holds.
Since $\cL$ satisfies (LT1) and (LT3), so does $\tL$ (by Remark \ref{rem:homotope}(i)).
Also, it follows from (b) and \eqref{eq:isograding0} that $\tL$ satisfies (LT2)(i).
To show that $\tL$ satisfies (LT2)(ii), let $(\al,\lm)\in \supp_{Q\times\Lm}(\tL)$ and
$\al\in \Dec$.  Then $(\al,\lm+\shom(\al))\in \supp_{Q\times\Lm}(\cL)$, so we can choose
$e_{\al}^{\lm+\shom(\al)} \in \dL{\al}{\lm+\shom(\al)}$ and
$f_{\al}^{\lm+\shom(\al)} \in \dL{-\al}{-\lm-\shom(\al)}$ as in (LT2)(ii) (for $\cL$).
Set $\tilde e_{\al}^\lm = e_{\al}^{\lm+\shom(\al)}$ and
$\tilde f_{\al}^\lm = f_{\al}^{\lm+\shom(\al)}$.
Then $\tL_{\al}^{\lm} = k\tilde e_{\al}^\lm$, $\tL_{-\al}^{-\lm} = k\tilde f_{\al}^\lm$, and
\[
[[\tilde e_{\al}^\lm,\tilde f_{\al}^\lm],  x_\beta]
= [[e_{\al}^{\lm+\shom(\al)},f_{\al}^{\lm+\shom(\al)}],  x_\beta]
= \langle \beta, \al^\vee\rangle x_\beta
\]
for $x_\beta\in \tL_\beta = \cL_\beta$, $\beta\in Q$.  Thus, $\tL$ satisfies (LT2)(ii).
To show that $\tL$ satisfies (LT4), let
\[\widetilde{\sem}_\al = \set{\lm\in\Lm \suchthat \tL_\alpha^{\lm}\ne 0} = \sem_\al - \shom(\al)\]
for $\al \in Q$.  Choose $\beta\in \Desh$. Then, by Lemma \ref{lem:suppLie}~(v),
$2\Lm \subseteq \sem_\beta$.
So, if $\lm\in \sem_\beta$, we have
$\lm = (\lm-\shom(\beta))+(2\shom(\beta) - \shom(\beta))
\in (\sem_\beta - \shom(\beta)) +  (\sem_\beta - \shom(\beta))
\in \widetilde{\sem}_\beta + \widetilde{\sem}_\beta$.  So
$\sem_\beta \subseteq \tilde\sem_\beta +  \tilde\sem_\beta$.
But, by Lemma \ref{lem:suppLie}~(v), we have
$\langle \sem_\beta \rangle = \Lm$. So $\langle \tilde\sem_\beta \rangle = \Lm$
and hence $\supp_\Lm(\tL) = \tL$.  Therefore $\tL$ satisfies (LT4) and we have (a).

``(b)$\Leftrightarrow$(c)''  The implication ``(b)$\Rightarrow$(c)'' is trivial.
For the converse we assume that (c) holds.
Let $\gamma\in \Deic$.  Then
$\gamma = w_{\al_1}\dots w_{\al_p} \al_{p+1}$ for some $\al_1,\dots,\al_p,\al_{p+1} \in \Pi$.
Let $g = g_1\dots g_p$, where
\[g_i = \theta_{\al_i}^{\shom(\al_i)} =
\exp(\ad(e_{\al_i}^{\shom(\al_i)}))\exp(\ad(-f_{\al_i}^{\shom(\al_i)})) \exp(\ad(e_{\al_i}^{\shom(\al_i)}))\]
for $1\le i \le p$.   Then, by \eqref{eq:Weylaut}, $g(\cL_{\al_{p+1}}) = \cL_\gamma$,
so $g(\tL_{\al_{p+1}}) = \tL_\gamma$.
But $e_{\al_i}^{\shom(\al_i)}$ and $f_{\al_i}^{\shom(\al_i)}$
are in $\tL^0$ for $1\le i \le p$, so $g(\tL^0) \subseteq \tL^0$.
Therefore, $g(\tL^0_{\al_{p+1}}) \subseteq g(\tL^0_{\gm})$.  Hence, since $\tL^0_{\al_{p+1}} \ne 0$,
we have $\tL^0_{\gm}\ne 0$.  So $\shom(\gamma) \in \sem_\gamma$, and (b) holds.
\end{proof}

\begin{definition}
\label{def:isotope}
Suppose that $\shom\in \Hom(Q,\Lm)$.  If $\shom$ satisfies the conditions (a)--(c) in Proposition
\ref{prop:isotope}, we say that $\shom$ is \emph{admissible} for $\cL$.  In that case,
we call the Lie torus $\cLs$ the $\shom$-\emph{isotope} of $\cL$.
\end{definition}

\begin{remark}
\label{rem:motivation2}
There are coordinatization theorems for the centreless Lie tori of each type (as mentioned
in the introduction),
and the notion of isotope for a centreless Lie torus corresponds in several
cases to a corresponding classical notion of isotope for the coordinate algebra \cite{AF}.
To give one example, Yoshii has shown that every centreless Lie torus of type $A_1$
is isomorphic to the Tits-Kantor-Koecher Lie algebra $\TKK(\cJ)$ constructed
from a Jordan torus $\cJ$ \cite{Y1}.  One can easily show that
the isotopes of the Lie torus $\TKK(\cJ)$ are up to bi-isomorphism the
Lie tori $\TKK(\cJ^{(u)})$, where $u$ runs over all nonzero homogeneous
(hence invertible) elements of $\cJ$ and where
$\cJ^{(u)}$ is the $u$-isotope of $\cJ$ \cite[Prop.~8.5]{AF}.  (See for example
\cite[\S 7.2]{Mc} for information about isotopes of Jordan algebras.)
\end{remark}

\begin{remark}
\label{rem:specify}  Suppose that $\Pi = \set{\al_1,\dots,\al_r}$ is a base
for $\De$.  By Proposition~\ref{prop:isotope}, to specify an admissible
$\shom\in \Hom(Q,\Lm)$ for $\cL$, and hence an isotope $\cLs$ of $\cL$,
one can arbitrarily choose $\lm_i\in \Lm_{\al_i}$ for $1\le i \le r$, and then
define $\shom\in \Hom(Q,\Lm)$ with $\shom(\al_i) = \lm_i$ for $1\le i \le r$.
\end{remark}

\begin{lemma}
\label{lem:isotopeequiv1}
Suppose that $\shom \in \Hom(Q,\Lm)$ is admissible for $\cL$. Then:
\begin{itemize}
\item[(i)] $\cL^{(0)} = \cL$.
\item[(ii)] If $t\in \Hom(Q,\Lm)$, then $t$ is admissible for
$\cLs$ if and only if $\shom+t$ is admissible for $\cL$.  Moreover,
in that case, $(\cLs)^{(t)} = \cL^{(\shom + t)}$.
\item[(iii)] $-\shom$ is admissible for $\cLs$ and
$(\cLs)^{(-\shom)} = \cL$.
\end{itemize}
\end{lemma}

\begin{proof}
(i) is clear.  For (ii),  we write
$\sem_\al^{(\shom)} := \set{\lm\in \Lm : (\cLs)_\al^\lm \ne 0} = \sem_\al - \shom(\al)$
for $\al\in Q$.  Then,
\begin{align*}
\text{$t$ is admissible for $\cLs$}
&\iff \text{$t(\al) \in \sem_\al^{(\shom)}$ for   $\al\in \Deic$}\\
&\iff \text{$t(\al) \in \sem_\al-\shom(\al)$ for  $\al\in \Deic$}\\
&\iff \text{$(\shom + t)(\al) \in \sem_\al$ for   $\al\in \Deic$}\\
&\iff \text{$\shom + t$ is admissible for $\cL$}
\end{align*}
The second statement in (ii) is clear.  Finally, (iii) follows from (i) and (ii).
\end{proof}

The following lemma is easily checked.

\begin{lemma}
\label{lem:isotopeequiv2}
Suppose that $\cL$ is a Lie $\Lm$-torus of type $\De$,
$\cL'$ is a Lie $\Lm'$-torus of type $\De'$, and
$\ph$ is an bi-isomorphism of $\cL$ onto $\cL'$
(see Definition \ref{def:biLT}). Suppose that
$\shom\in \Hom(Q,\Lm)$ and $\shom'\in \Hom(Q',\Lm')$ with
$\ph_e \shom = \shom' \ph_r$.
Then  $\shom$ is admissible for $\cL$ if and only if
$\shom'$ is admissible for $\cL'$.  Moreover, in that case,
$\ph$ is an bi-isomorphism of $\cLs$ onto~${\cL'}^{(\shom')}$
(with the same root-grading isomorphism and
the same external-grading isomorphism).
\end{lemma}

\begin{definition}
\label{def:isotopy}
Suppose that $\cL$ is a Lie $\Lm$-torus of type $\De$ and
$\cL'$ is a Lie $\Lm'$-torus of type $\De'$.
We say that $\cL$ is \emph{isotopic} to $\cL'$, written
$\cL\sim \cL'$, if some isotope $\cLs$ of $\cL$ is bi-isomorphic
to $\cL'$.
\end{definition}

It follows easily from Lemmas \ref{lem:isotopeequiv1} and \ref{lem:isotopeequiv2} that
$\sim$ is an equivalence relation on the class of Lie tori.

\section{Multiloop realization}
\label{sec:multrealization}

In rest of the paper  we investigate multiloop realizations
of centreless Lie tori of nullity $n$.  If $n = 0$, any such Lie torus is a finite dimensional split simple Lie algebra (see Remark \ref{rem:nullity0}), and we do not have anything to add to the theory of these algebras.  Thus,
\emph{for the rest of the paper we assume that $n$ is an integer $\ge 1$}.

\emph{We also suppose  for the rest of the paper that $k$ is algebraically closed} (of characteristic~0).
Thus, $k^\times$  contains a primitive root $\ell^\text{th}$
of unity $\zeta_\ell$ for all positive integers $\ell$. We assume that
we have made a fixed \emph{compatible} choice of these roots of unity in  $k^\times$
in the sense that
\[\zeta_{m \ell}^m = \zeta_\ell\]
for all $\ell, m\ge 1$.

\subsection{Multiloop algebras}
\label{subsec:multiloop}

We first recall  the definitions and the results that we
will need from \cite{ABFP} regarding multiloop algebras  in general.
In  order to do this conveniently we first introduce some notation.

\begin{notation}
\label{not:action} \

(i) Suppose that $G$ is a group.
We let
\[\cfo_n(G) = \set{\bsg = (\sg_1,\dots,\sg_n)\in G^n  \suchthat  \order{\sg_i} <  \infty,\ \sg_i\sg_j = \sg_j\sg_i \text{ for all } i,j}\]
denote the set of all length $n$ sequences of commuting finite order elements
of elements of $G$.
If $P=(p_{ij})\in \Mat_n(\bbZ)$ and
$\bsg = (\sg_1,\ldots,\sg_n)\in \cfo_n(G)$, we set
\[
\bsg^P=(\prod\nolimits_{i=1}^{n}\sigma_{i}
^{p_{i1}},\ldots,\prod\nolimits_{i=1}^{n}\sigma_{i}^{p_{in}})\in\cfo_n(G).
\]
One sees that
\begin{equation}
\label{eq:Mnaction}
(\bsg^{P})^{P'}=\bsg^{PP'}
\end{equation}
for $P,P' \in \Mat_n(\bbZ)$. In particular,
this gives a right action of $\GL_n(\bbZ)$ on $\cfo_n(G)$.
If $\bsg\in\cfo_n(G)$, we let
\[\langle \bsg \rangle = \langle \sg_1,\dots,\sg_n\rangle\]
denote the subgroup of $G$ generated by $\sg_1,\dots,\sg_n$.
Then, if $P\in \Mat_n(\bbZ)$, it is clear that
$\langle \bsg^P \rangle \subseteq \langle \bsg \rangle$; so if
$P \in \GL_n(\bbZ)$, we have
\begin{equation}
\label{eq:bsggen}
\langle \bsg^P \rangle = \langle \bsg \rangle.
\end{equation}
Let
\[\boldone = (1,\dots,1).\]
in  $G^n$. Then if $\bsg\in \cfo_n(G)$ and
$\boldm \in \bbZ^n_+$, where $\bbZ_+$  denotes the set of positive integers,  we write
\[\bsg^\boldm = \boldone\]
to mean that $\sg_i^{m_i} = 1$ for all $i$.

(ii) If $\cA$ is an algebra, we use the notation in (i) for the group
$G=\Aut_k(\cA)$. For brevity we write
\[\cfo_n(\cA) = \cfo_n(\Aut_k(\cA)).\]
That is $\cfo_n(\cA)$ is the set of length $n$ sequences
of  commuting finite order automorphisms of $\cA$.
We have a right action of $\GL_n(\bbZ)$ on
$\cfo_n(\cA)$ as in (i).

(iii) If $\ph : \cA \to \cA'$ is an isomorphism of algebras and
$\bsg = (\sg_1,\ldots,\sg_n)\in \cfo_n(\cA)$, we set
\[\ph\bsg\ph^{-1} = (\ph\sg_1\ph^{-1},\ldots,\ph\sg_n\ph^{-1})\in \cfo_n(\cA').\]
\end{notation}

\begin{definition}
\label{def:multiloop}
Suppose $\cA$ is an (ungraded) algebra
over~$k$,  $\bsg = (\sg_1,\dots,\sg_n)\in \cfo_n(\cA)$, and
$\boldm = (m_1,\dots,m_n) \in \bbZ^n_+$ with $\bsg^\boldm = \boldone$.
Let  $\bLm = \bbZ/(m_1)\oplus \dots \oplus \bbZ/(m_n)$
and let $\lm = (\ell_1,\dots,\ell_n) \mapsto \blm = (\modell_1,\dots,\modell_n)$
be the canonical group homomorphism from $\Zn$ onto $\bLm$, where
$\modell_j = \ell_j + m_j\bbZ$ for each $j$.
We give the algebra $\cA$ a $\bLm$-grading,
called the \emph{$\bLm$-grading of $\cA$ determined by $\bsg$}, by setting
\begin{equation*}
 \cA^\blm = \set{u\in \cA \suchthat \sg_j u
= \zeta_{m_j}^{\ell_j} u \text{ for } 1\le j\le n}
\end{equation*}
for $\blm = (\modell_1,\dots,\modell_n) \in \bLm$.
Let
$k[z_1^{\pm 1},\dots,z_n^{\pm1}] = \sum_{\lm\in\Zn} k z^\lm$
be the algebra of Laurent polynomials
over~$k$, where $z^\lm = z_1^{\ell_1}\dots z_n^{\ell_n}$ for $\lm = (\ell_1,\dots,\ell_n)\in\Zn$.
Let
\[
\textstyle
\loopm(\cA,\bsg) = \sum_{\lm\in \Zn}
\cA^\blm \ot z^\lm
\]
in the  algebra $\cA \ot_k k[z_1^{\pm 1},\dots,z_n^{\pm 1}]$. Then
$\loopm(\cA,\bsg)$ is a subalgebra of $\cA \ot_k k[z_1^{\pm 1},\dots,z_n^{\pm
1}]$. We define a $\Zn$-grading on $\loopm(\cA,\bsg)$  by setting
\begin{equation}
\label{eq:multgrad}
\loopm(\cA,\bsg)^\lm =
\cA^\blm\ot z^\lm
\end{equation}
for all  $\lm\in \Zn$.
We call the $\Zn$-graded algebra
$\loopm(\cA,\bsg)$ the \textit{multiloop algebra}
of $\bsg$ (based on $\cA$ and relative
to  $\boldm$).
\end{definition}

By \cite[Lemma 3.2.4]{ABFP},  we have the following:

\begin{lemma}
\label{lem:multprop} Suppose that $\boldm$, $\cA$ and
$\bsg$ are  as in Definition \ref{def:multiloop}.
Then
\[\langle \supp_{\Zn}(\loopm(\cA,\bsg))\rangle =\Zn \iff
\order{\langle \sg_{1},\ldots,\sg_n\rangle} =m_{1}\cdots m_n.
\]
\end{lemma}

The following result, which is part of \cite[Corollary 8.3.5]{ABFP},
tells us which graded algebras have multiloop realizations
based on finite dimensional simple algebras.

\begin{theorem}
\label{thm:realmult}
Let $\cB$ be a $\Lm$-graded algebra, where
$\Lm$ is a free abelian group of finite rank  $n\ge 1$.
Then $\cB$ is isograded-isomorphic
to $\loopm(\cA,\bsg)$ for some finite dimensional simple (ungraded) algebra $\cA$,
some  $\bsg \in \cfo_n(\cA)$  and
some $\boldm \in \bbZ^n_+$ with $\bsg^\boldm = \boldone$
if and only  if $\cB$ is graded-central-simple, $\cB$
is fgc and $\Gm(\cB)$ has finite index in $\Lm$.
\end{theorem}

We also have necessary and sufficient conditions for
isograded-isomorphism of multiloop algebras $\loopm(\cA,\bsg)$ that  satisfy
$\langle \supp_{\Zn}(\loopm(\cA,\bsg)) = \Zn$.

\begin{theorem}
\label{thm:isomult} Suppose
$\cA$ is an (ungraded) central-simple algebra
over~$k$,  $\bsg =  (\sg_1,\dots,\sg_n) \in \cfo_n(\cA)$,
$\boldm = (m_1,\dots,m_n) \in \bbZ^n_+$ with $\bsg^\boldm = \boldone$,
and $\left\vert\langle\sg_{1},\ldots,\sg_n\rangle\right\vert
=m_{1}\cdots m_n$. Suppose also that $\boldm'$,
$\cA'$ and $\bsg'$ satisfy the same assumptions.
Then $\loopm(\cA,\bsg)$ is isograded-isomorphic to $\loopmp(\cA',\bsg')$
if and only if there exists a matrix $P =(p_{ij}) \in\GL_n(\bbZ)$ and an algebra isomorphism
$\ph : \cA\to \cA'$ such that
\begin{equation*}
      \sg'  = \ph \bsg^P \ph^{-1}.
\end{equation*}
\end{theorem}

\begin{proof}  Note that by Lemma \ref{lem:multprop}
we have $\langle \supp_{\Zn}(\loopm(\cA,\bsg))\rangle =\Zn$
as well as the corresponding primed statement.  The theorem
then is the second statement of \cite[Theorem 8.3.2(ii)]{ABFP}.
\end{proof}

\begin{remark}
\label{rem:GP} The action  $(\sg,P) \mapsto \sg^P$ of $\GL(n,\bbZ)$ on
$\cfo_n(\cA)$ also plays an important role in the study of the isomorphism problem
for multiloop algebras regarded as (ungraded) algebras  \cite[Lemma 5.3]{GP}.
\end{remark}

\subsection{The construction of multiloop Lie tori}
\label{subsec:construction}

We next  want to describe a multiloop construction of a
centreless Lie torus $\LT(\fs,\sg,\fh)$ that we will call a multiloop Lie $\Zn$-torus.
For technical reasons we begin with a more general construction of a $Q\times\Zn$-graded Lie algebra
that we denote by $\LT_\boldm(\fs,\sg,\fh)$.    (LT in this notation  is once again
an abbreviation
for Lie torus.)

If $\fs$ is an algebra and $\bsg = (\sg_1,\dots,\sg_n) \in \cfo_n(\fs)$, we let
\[\fs^\bsg = \set{x\in \fs \suchthat \sg_i(x) = x \text{ for } 1\le i \le n}\]
denote the fixed point subalgebra for $\langle\bsg\rangle$ in $\fs$.

\begin{construction}
\label{con:LTconstruction}
Let $\fs$ be a finite dimensional simple Lie algebra
over $k$; let   $\bsg= (\sg_1,\dots,\sg_n) \in \cfo_n(\fs)$;
let $\fh$ be an ad-diagonalizable subalgebra
of $\fs$ that is contained in $\fs^\bsg$;
and let $\boldm = (m_1,\dots,m_n) \in \bbZ^n_+$ with $\bsg^\boldm = \boldone$.
We  now construct a $Q\times \Zn$-graded algebra $\LT_\boldm(\fs,\bsg,\fh)$
from this data, where
\begin{equation}
\label{eq:DeQ}
\De = \De(\fs,\fh) \andd Q = \spann_\bbZ(\De).
\end{equation}
(However, with our assumptions so far, $\De$ is not necessarily a root system.)

Set  $\bLm =\bbZ/(m_1)\oplus \dots \oplus \bbZ/(m_n)$.
Then $\fs$ is $Q$-graded using the root space decomposition
of $\fs$ relative to the adjoint action of $\fh$ (see
Definition \ref{def:rootspace}), and $\fs$ is $\bLm$-graded using the
$\bLm$-grading determined by $\bsg$ (see Definition \ref{def:multiloop}).
These gradings are compatible and so they determine a $Q\times \bLm$-grading
$\fs = \oplus_{(\al,\blm)\in Q\times\bLm} \fs_\al^\blm$ on $\fs$.
Define the algebra
\[\LT_\boldm(\fs,\bsg,\fh) = \textstyle \sum_{\lm \in \Zn}
\fs^{\blm}\ot z^\lm
\]
in $\fs\ot k[z_1^{\pm 1},\dots,z_n^{\pm 1}]$, and give this algebra a $Q\times \Zn$-grading
by setting
\begin{equation}
\label{eq:LTgrading}
\LT_\boldm(\fs,\bsg,\fh)_\al^\lm = \ds{\al}{\blm}\ot z^\lm
\end{equation}
for $\al\in Q$ and $\lm \in \Zn$.
Thus, the $Q$-grading on $\LT_\boldm(\fs,\bsg,\fh)$ is the root space
decomposition of
$\LT_\boldm(\fs,\bsg,\fh)$ relative to the ad-diagonalizable subalgebra $\fh\ot 1$.
As a
$\Zn$-graded algebra $\LT_\boldm(\fs,\bsg,\fh)$ is the multiloop
algebra $\loopm(\fs,\bsg)$ defined in Definition \ref{def:multiloop}.
\end{construction}

\begin{remark}
\label{rem:LTconstruction}
Suppose that $\fs$,  $\bsg$, $\fh$ and $\boldm$
are as in Construction \ref{con:LTconstruction}.

(i) For technical reasons, we only assumed in the construction that
$\sg_i$ has period $m_i$.  However, we will see below that a necessary condition
for $\LT_\boldm(\fs,\bsg,\fh)$ to be a Lie torus is that
$\order{\sg_i} = m_i$ for all $i$,  at which point we will
simplify notation by deleting the subscript  $\boldm$.

(ii) One choice of $\fh$ satisfying the requirements in the construction is a Cartan subalgebra of
$\fs^\bsg$ \cite[Thms. III.10 and III.17]{J}.  Conversely, we will see below that
a necessary condition for $\LT_\boldm(\fs,\bsg,\fh)$ to be a Lie torus is that
$\fh$  is a Cartan subalgebra of $\fs^\bsg$.  Again, for technical reasons we do not make
that assumption
at the outset.

(iii) $\LT_\boldm(\fs,\bsg,\fh)$ is fgc and centreless.  (In fact it is fgc and $\Zn$-graded-central-simple
by Theorem \ref{thm:realmult}.)
\end{remark}

In order to describe conditions  for $\LT_\boldm(\fs,\bsg,\fh)$ to be a
Lie torus, it will be convenient to introduce the following terminology for modules.

\begin{definition}
\label{def:condM}
If $\fg$ is a finite dimensional simple Lie algebra, $\fh$ is a Cartan
subalgebra of $\fg$  and
$V$ is a finite dimensional module for $\fg$, we say that $V$
\emph{satisfies condition \emph{(M)}} if
\begin{itemize}
\item[(M)]  $V$ is irreducible of dimension $>1$ and
the weights of $V$ relative to $\fh$ are contained in $\De(\fg,\fh)_\mathrm{en}$,
where $\fh$ is a Cartan subalgebra of $\fg$.
\end{itemize}
(See \eqref{not:rootsystem} for the  notation $\De(\fg,\fh)_\mathrm{en}$.)
Note  that the validity of condition (M) does not depend on the
choice of $\fh$.
\end{definition}

It is easy to list the modules satisfying this condition:

\begin{lemma}
\label{lem:3mod}
Suppose that $\fg$ is a finite dimensional Lie simple algebra with Cartan
subalgebra $\fh$.  We fix a base for the root system $\De_\fg = \De(\fg,\fh)$.
For a dominant integral weight $\mu$ for $\fg$ (relative to $\fh$ and the chosen base),
let $V(\mu)$ be the irreducible module with highest weight $\mu$.  Let
$\theta$ and $\theta_\mathrm{sh}$ denote the highest root and the highest short root
of $\fg$ respectively.

If $V$ is a finite dimensional $\fg$-module, then $V$ satisfies condition (M) if and only if
$V$ is isomorphic to one of the following:
\begin{itemize}
\item[(i)] The adjoint module $V(\theta)$,
\item[(ii)] The little adjoint module $V(\theta_\mathrm{sh})$ if $\De_\fg$ is not simply laced,
\item[(iii)] The symmetric module $V(2\theta_\mathrm{sh})$ if $\De_\fg$ is of type B$_\ell$ with $\ell \ge 1$.
\end{itemize}
Moreover, if $V$ satisfies condition (M), then
the multiplicity of each nonzero weight in $V$ is 1.
Also, if $V$ and $V'$ satisfy condition (M), then $V$ and $V'$ have a nonzero weight in common.
\end{lemma}

\begin{proof}
The first statement follows from the fact
that the orbits under the Weyl group of $\De_\fg$ acting
on $(\De_\fg)_\mathrm{en}$ are represented
by $0$, $\theta$,  $\theta_\mathrm{sh}$ (if $\De_\fg$ is not simply laced),
and $2\theta_\mathrm{sh}$ (if $\De_\fg$ is of type B$_\ell$ with $\ell \ge 1$).
The last two statements follow from  the first statement
and a calculation of the weights and multiplicities
of the modules in (i)--(iii) (see for example \cite[Chapter 2]{ABG}).
\end{proof}

In Construction \ref{con:LTconstruction}, if $\fg = \fs^\bsg$ and
$\blm\in \bLm$ ,
then $\fs^\blm$ is a $\fg$-module under the adjoint  action.

\begin{proposition}
\label{prop:N}
Let $\fs$, $\bsg$,  $\fh$ and $\boldm$ be as
Construction \ref{con:LTconstruction}, and define $\De$ and $Q$  by \eqref{eq:DeQ}.
Put $\fg = \fs^\bsg$ and $\De_\fg = \De(\fg,\fh)$.
Then, the following statements are equivalent:
\begin{itemize}
\item[(a)] $\De$ is a  finite irreducible root system in $\fh^*$ and $\LT_\boldm(\fs,\bsg,\fh)$
is a Lie $\Zn$-torus of type $\De$.
\item[(b)] $\fh$ is a Cartan subalgebra of $\fg$, $m_i = \order{\sg_i}$ for each $i$,
and $\bsg$ satisfies the following conditions:
\begin{itemize}
\item[(A1)] $\fg$ is a simple Lie algebra.
\item[(A2)] If $0\ne \blm\in \supp_{\bLm}(\fs)$,  then,
$\fs^\blm \simeq U^\blm \oplus V^\blm$ as a $\fg$-module,
where $\fg$ acts trivially on $U^\blm$ and either $V^\blm = 0$ or
$V^\blm$ satisfies condition (M).
\item[(A3)] $\order{\langle \sg_1,\dots,\sg_n \rangle} =
\prod_{i=1}^n \order{\sg_i}$.
\end{itemize}
\end{itemize}
Moreover, if (a) and (b) hold, then $\LT_\boldm(\fs,\bsg,\fh)$ satisfies
(LT5), $(\fg\ot 1,\fh\ot 1)$ is the root grading pair for $\LT_\boldm(\fs,\bsg,\fh)$,
and
\begin{equation}
\label{eq:propN2}
\Dei = \De_\fg.
\end{equation}
In fact, in that case,
\begin{equation}
\label{eq:propN3}
\De = \De_\fg \orr \De =  (\De_\fg)_\mathrm{en},
\end{equation}
and $\De \ne \De_\fg$ if and only if $\fg$ has type B$_\ell$ for some $\ell \ge 1$
and  $V^\blm$
is isomorphic to the symmetric  module  for some $0\ne \blm\in\supp_\bLm(\fs)$.
\end{proposition}

\begin{proof} Let $\cL = \LT_\boldm(\fs,\bsg,\fh)$.

``(b)$\Rightarrow$(a)'' Assume that (b) holds. Now $\De$ is the union
of the sets of weights of the $\fg$-modules $\fs^\blm$, $\blm\in
\supp_\bLm(\fs)$. So it follows from (A1), (A2) and Lemma
\ref{lem:3mod}  that we have the last statement of the lemma. It
then follows from \eqref{eq:propN3} that \eqref{eq:propN2}  holds.

We claim next that  $\fs$ is a Lie $\bLm$-torus of type $\De$ satisfying
(LT5). So we must
check (LT1)--(LT5) for $\fs$.
First since $\De = \De(\fs,\fh)$ and since
$\Dei = \De_\fg$, (LT1) and (LT2)(i) are clear.
Thus, since $\fs$ is simple, (LT3) also holds for $\fs$ by Remark
\ref{rem:LT3}.
Next, by assumption (A3) and Lemma \ref{lem:multprop}, we have
$\langle \supp_\Zn(\loopm(\fs,\bsg))\rangle = \Zn$,
so, by \eqref{eq:multgrad},  $\langle \supp_\bLm(\fs)\rangle = \bLm$.
Thus, (LT4) holds for $\fs$.
Also, again since
$\De = \De(\fs,\fh)$, (LT5) holds for $\fs$.
So we are left only with (LT2)(ii) for $\fs$.
For this, note first that by (A2) and Lemma \ref{lem:3mod},
if $\al\in \Dec$ and $0\ne \blm\in \supp_{\bLm}(\fs)$,
the  multiplicity of
the $\al$-weight space in the module $\fs^\blm$ is  $\le 1$.
But this is also true if $\blm = \bar 0$.   So
\begin{equation}
\label{eq:Kdim1}
\dim(\ds{\al}{\blm}) \le 1
\end{equation}
for $\al\in \Dec$, $\blm\in \bLm$. Next let $\form_{\fs}$ be the
Killing form on $\fs$.  Then $\form_{\fs}$ is a nondegenerate
invariant symmetric bilinear form on $\fs$. Furthermore, the
automorphisms $\sg_1,\dots,\sg_n$ preserve this form and
$\ad_\fs(h)$ is a skew transformation relative to this form for
$h\in \fh$.  Hence it follows that
\begin{equation}
\label{eq:paired}
(\ds{\al}{\blm},\ds{\beta}{\bar \mu})_\fs = 0
\quad\text{unless $\al+\beta = 0$  and  $\blm + \bar\mu = 0$}
\end{equation}
for $\al,\beta\in Q$ and $\blm,\bar\mu\in \bLm$. Now let
$(\al,\blm)\in \supp_{Q\times\bLm}(\fs)$ with $\al\in \Dec$. Then,
by \eqref{eq:paired} and the nondegeneracy of $\form_{\fs}$, we have
$(-\al,-\blm)\in \supp_{Q\times\bLm}(\fs)$.  So we can choose $0\ne
e_{\al}^{\blm} \in \ds{\al}{\blm}$ and $0\ne f_{\al}^{\blm} \in
\ds{-\al}{-\blm}$, in which case we have $\ds{\al}{\blm} =
ke_{\al}^{\blm}$ and $\ds{\al}{\blm} = kf_{\al}^{\blm}$ by
\eqref{eq:Kdim1}. Next we set $h_\al^{\blm} =
[e_{\al}^{\blm},f_{\al}^{\blm}] \in \ds{0}{\bar 0} = \fh$. Then, if
$h\in \fh$, we have
\[(h,h_\al^{\blm})_\fs = (h,[e_{\al}^{\blm},f_{\al}^{\blm}])_\fs =
([h,e_{\al}^{\blm}],f_{\al}^{\blm})_\fs
 = \al(h)(e_{\al}^{\blm},f_{\al}^{\blm})_\fs\]
Thus, $h_\al^{\blm}$ is a nonzero multiple of $\al^\vee$, so, adjusting
$f_{\al}^{\blm}$ by a nonzero scalar, we can assume that
$h_\al^{\blm} = \al^\vee$.  Therefore we have
(LT2)(ii) for $\fs$.  Thus, as claimed,
$\fs$ is a Lie $\bLm$-torus of type $\De$ satisfying
(LT5).

To see that
$\cL$ is a $\Zn$-Lie torus of type $\De$ that also satisfies (LT5),
we must check the axioms (LT1)--(LT5) for $\cL$.
Indeed,  by \eqref{eq:LTgrading}, if $(\al,\lm)\in Q\times \Zn$, we have
\[(\al,\lm)\in \supp_{Q\times\Zn}(\cL) \iff (\al,\blm)\in \supp_{Q\times\bLm}(\fs). \]
Using this fact, the axioms (LT1), (LT2), (LT4) and (LT5) for $\cL$
follow easily from the corresponding axioms for $\fs$.  Finally,
since $\cL$ is a graded-simple $\Zn$-graded algebra by
Theorem \ref{thm:realmult}, $\cL$ satisfies
(LT3) by Remark \ref{rem:LT3}. Thus, as claimed,
$\cL$ is a Lie $\Zn$-torus  of type $\De$ satisfying
(LT5).

So we have established (a)  as well as \eqref{eq:propN2} and \eqref{eq:propN3}.
Also, by Proposition \ref{prop:Lietorusbasic}(vi), the root grading
pair  for $\cL$ is $(\cL^0,\cL_0^0) = (\fs^\modz\ot 1,\fs_0^\modz \ot 1) = (\fg \ot 1,\fh\ot 1)$.

``(a)$\Rightarrow$(b)'' Suppose (a) holds.  By Proposition \ref{prop:Lietorusbasic}(vi),
the grading pair for $\cL$ is $(\cL^0,\cL_0^0) = (\fg\ot 1,\fs_0^\modz\ot 1)$.
So, by Proposition \ref{prop:Lietorusbasic} (ii) and (iii), $\fg$ is simple, $\fs_0^\modz$ is a Cartan
subalgebra of $\fg$, and $\dim(\fs_0^\modz) = \rank(Q)$.  But, since
$\De$ is a root system in $\fh^*$, we have
$\dim(\fh) = \rank(Q)$. Hence, since $\fs_0^\modz\supseteq\fh$, we have $\fs_0^\modz =\fh$.
Thus, $\fg$ is simple and $\fh$ is a Cartan subalgebra of $\fg$.

By (LT4), $\Zn = \langle \supp_\Zn(\cL)\rangle$, so, by Lemma \ref{lem:multprop},
$\order{\langle \sg_{1},\ldots,\sg_n\rangle} =m_{1}\cdots m_n$.
Thus we have $m_i =\order{\sg_i}$, and (A1) and (A3) hold.

Next, if $\al\in \Deic$, then,  by (LT2)(i), we have
$0\ne \cL_\al^0 = \fs_\al^\modz\ot 1 = \fg_\al\ot 1$, so $\al\in \De_\fg$.  Therefore,
$\Dei \subseteq \De_\fg$, so  $\Dei \subseteq \De_\fg\subseteq \De$.
Since $\De_\fg$ is reduced, we have $\Dei = \De_\fg$.  So
$\De = \De_\fg$  or $\De = (\De_\fg)_\mathrm{en}$.

Now to prove (A2) let $0\ne \blm \in \supp_{\bLm}(\fs)$.  Since
$\De \subseteq (\De_\fg)_\mathrm{en}$, the weights
of the $\fg$-module $\fs^\blm$ are contained
in $(\De_\fg)_\mathrm{en}$.
It therefore remains to show that this $\fg$-module has at most one
irreducible component of dimension $\ge 2$.
Suppose the contrary.
Then, by Lemma \ref{lem:3mod}, the $\al$-weight space
in $\fs^\blm$ has dimension $>1$ for some $\al\in \Dec$.
So $\dim \ds{\al}{\blm} \ge 2$ and hence
(by \eqref{eq:LTgrading}) $\dim \dL{\al}{\lm} \ge 2$,
contradicting   (LT2)(ii).
\end{proof}

\begin{definition}
\label{def:multLT}
Suppose that $\fs$ is a finite dimensional simple Lie algebra
over $k$, $\bsg = (\sg_1,\dots,\sg_n)$ is an sequence in $\cfo_n(\fs)$
satisfying (A1)--(A3), and $\fh$ is a Cartan subalgebra
of $\fg = \fs^\bsg$.  Let $\boldm = (m_1,\dots,m_n)$,
where $m_i = \order{\sg_i}$, $\De = \De(\fs,\fh)$ and $Q = \spann_\bbZ(\De)$.
Since $\boldm$ is determined by $\bsg$, we simplify notation and write
\[\LT(\fs,\bsg,\fh) := \LT_\boldm(\fs,\bsg,\fh).\]
Then, by Proposition \ref{prop:N} and Remark
\ref{rem:LTconstruction}(iii),
$\LT(\fs,\bsg,\fh)$ is an fgc centreless $\Zn$-Lie torus of type $\De$.
We call $\LT(\fs,\bsg,\fh)$ the \emph{multiloop Lie $\Zn$-torus}
(or simply the multiloop Lie torus)  determined by  $\fs$, $\bsg$ and $\fh$.
\end{definition}

\begin{remark}
\label{rem:Cartaninvariance} Suppose that we have the assumptions
and notation of Definition \ref{def:multLT}. Suppose that $\fh'$ is
another Cartan subalgebra of $\fg$. Note that $\LT(\fs,\bsg,\fh)$
and $\LT(\fs,\bsg,\fh')$ are equal as $\Zn$-graded algebras.  Hence,
by Proposition \ref{prop:torusequiv}, $\LT(\fs,\bsg,\fh)$ and
$\LT(\fs,\bsg,\fh')$ are bi-isomorphic as Lie tori.  Thus, up to
bi-iso\-mor\-phism, the Lie torus $\LT(\fs,\bsg,\fh)$ does not depend on
the choice of the Cartan subalgebra $\fh$ of $\fg$.
\end{remark}

\begin{example}
\label{ex:trivLT}
Suppose that $\fs$ is a finite dimensional simple Lie  algebra,
$\boldone = (1,\dots,1)$, and $\fh$ is a Cartan subalgebra
of  $\fs$.  Then, $\LT(\fs,\boldone,\fh) = \fs \ot k[z_1^{\pm1 },\dots,z_n^{\pm 1}]$
is a multiloop Lie torus called the  \emph{untwisted multiloop Lie $\Zn$-torus}
determined by $\fs$ and $\fh$.
\end{example}

\subsection{The realization theorem}

We can now prove  our first main result.

\begin{theorem}\emph{[Realization Theorem for Lie tori]}
\label{thm:realization}
Suppose  that  $k$ is an algebraically closed field of characteristic 0, and $n\ge 1$.
Then a centreless Lie $\Lm$-torus $\cL$ of nullity $n$ is bi-isomorphic to
a multiloop Lie $\Zn$-torus (as defined in Definition~\ref{def:multLT}) if and only if
$\cL$ is fgc.
\end{theorem}

\begin{proof}  The necessity of  the fgc condition is noted in Remark \ref{rem:LTconstruction}
and so we need only prove the sufficiency.  Let $\Lm$ be a free abelian group of rank $n$ and let
$\cL$ be a fgc centreless Lie $\Lm$-torus of type $\De$.
We let $Q = \spann_\bbZ(\De)$, so $\cL$
is a $Q\times\Lm$-graded algebra.
Replacing $\De$ by $\Dei$
if necessary we can assume (see Remark \ref{rem:LT5}) that
$\cL$ satisfies (LT5).  That is
\begin{equation}
\label{eq:suppQL}
\De = \supp_Q(\cL).
\end{equation}

Now, by Proposition \ref{prop:fgc}, $\Lm/\Gm_\Lm(\cL)$ is finite.
Also by Proposition \ref{prop:Lieprop}(i), $\cL$ is a graded-central-simple
$\Lm$-graded algebra.
Therefore by Theorem \ref{thm:realmult},
the $\Lm$-graded
algebra $\cL$ is isograded-isomorphic to the $\Zn$-graded algebra
$\loopm(\fs,\bsg)$ for some finite dimensional simple Lie algebra $\fs$,
some  $\bsg \in \cfo_n(\fs)$   and some
$\boldm\in \bbZ_+^n$ with $\bsg^\boldm = \boldone$.

So we have an algebra isomorphism $\ph: \cL \to \loopm(\fs,\bsg)$ and a group isomorphism
$\nu : \Lm \to \Zn$ such that $\ph(\cL^\lm) = \loopm(\fs,\bsg)^{\nu(\lm)}$ for $\lm\in \Lm$.
We use $\ph$ to transfer the $Q$-graded structure of $\cL$ to $\loopm(\fs,\bsg)$.
Then $\loopm(\fs,\bsg)$ becomes a Lie $\Lm$-torus of type $\De$.
We now use $\nu$ to identify $\Lm$ with $\Zn$ and we use $\ph$ to identify
\[\cL = \loopm(\fs,\bsg)\]
as a $Q\times \Lm$-graded algebra.

Next let $(\fg,\fh)$ be the root grading pair for $\cL$.
Then, by Proposition \ref{prop:Lietorusbasic}(vi),
$\fg = \cL^0 = \fs^\modz \ot 1$, so we can identify
$\fg$ with the subalgebra $\fs_\modz$ of $\fs$
(by identifying $x\ot 1$  with $x$).

Now we use the linear isomorphism $\widetilde{\phantom w}$ in Proposition \ref{prop:Lietorusbasic}(iii)
to identify $\De$ with a finite irreducible root system in $\fh^*$ such that
$\Dei = \De(\fg,\fh)$ and
\begin{equation*}
\cL_\al = \set{x\in \cL
\suchthat [h,x] = \al(h)x \text{ for } h\in \fh}
\end{equation*}
for $\al\in Q$.  Thus, by \eqref{eq:suppQL}, $\De = \De(\cL,\fh) = \De(\fs,\fh)$.
Our result now follows from Proposition \ref{prop:N}
((a)$\Rightarrow$(b)).
\end{proof}

\subsection{The case $n=1$}
\label{subsec:n=1}

In this subsection,  we show that when $n= 1$ conditions (A1) and (A2) are equivalent
to the condition that $\sg$ is a diagram automorphism.  ((A3) is automatic when $n=1$.)
So, when $n=1$, Theorem \ref{thm:realization} coincides with the classical realization theorem
for  affine Kac-Moody Lie algebras (see Remark \ref{rem:Kac} below).

We first recall the notion of a diagram automorphism
of a finite dimensional simple Lie algebra $\fs$.
If $\ft$ is a Cartan subalgebra  of $\fs$, $\set{\al_i}_{i=1}^r$ is a base for $\De(\fs,\ft)$,
and $0\ne e_i\in \fs_{\al_i}$ for $1\le i \le r$, then the pair $(\set{\al_i}$, $\set{e_i})$
is called an \emph{\'epinglage}
of $(\fs,\ft)$ \cite[chap.~VIII, \S~4, n$^\circ$~1]{B2}.  A \emph{diagram automorphism} of $\fs$ is an automorphism $\sg$ of
$\fs$ that stabilizes $\ft$ and $\set{e_i}$ for some choice
of Cartan subalgebra $\ft$ and \'epinglage $(\set{\al_i}, \set{e_i})$, in which case
$\sg$ is called a diagram automorphism of $\fs$ relative to $\ft$
and $(\set{\al_i}, \set{e_i})$.
Moreover, any diagram automorphism of $\fs$ is conjugate in $\Aut_k(\fs)$ to one relative to a fixed
Cartan subalgebra and \'epinglage  \cite[chap.~VIII, \S~5, prop.~5]{B2}.

\begin{proposition}
\label{prop:n=1}   A finite order  automorphism $\sg$ of a finite dimensional
simple Lie algebra $\fs$ satisfies (A1) and (A2) if and only if
$\sg$ is  a diagram automorphism.
\end{proposition}

\begin{proof}  First we fix some notation. Let
$\set{e_\al,h_i} = \set{e_\al: \al\in \De(\fs,\ft)^\times}\cup \set{h_i : 1\le i \le r}$
be a Chevalley basis for $\fs$. Set $\ft = \sum k h_i$, a Cartan subalgebra of $\fs$. Let
$\set{\al_i}_{i=1}^r$ be a base for $\De(\fs,\ft)$, in which case
$(\set{\al_i}$, $\set{e_{\al_i}})$ is an \'epinglage of $(\fs,\ft)$.

Now, by \cite[Proposition~4.29]{KW},
replacing $\sg$ by a conjugate in $\Aut_k(\fs)$, we may assume that
\[\sg =  \tau \hsg\]
where  $\tau,\hsg\in \Aut_k(\fs)$, $\tau \hsg = \hsg \tau$, $\hsg$ is a diagram automorphism of $\fs$ relative to
$\ft$ and $(\set{\al_i}$, $\set{e_{\al_i}})$, $\tau$ fixes
$\ft$ pointwise, and $\tau$ is diagonal relative to the Chevalley basis
$\set{e_\al,h_i}$ of $\fs$.
We will show that under these circumstances
\begin{equation}
\label{eq:n=1claim}
\text{$\sg$ is conjugate to $\hsg$ in $\Aut_k(\fs)$} \iff \text{$\sg$ satisfies (A1) and (A2)}.
\end{equation}
This suffices to prove the proposition.

We begin by reducing the proof of \eqref{eq:n=1claim} to the case when $k = \bbC$.
(We do this since we wish to use Kac's classification of finite order automorphisms
of $\fs$, which has been carried out only in that case (see \cite[Chap.~8]{K} or
\cite[\S 3.6--3.11]{GOV}).)   For this, let $\bar \bbQ$ be an  algebraic
closure of the rational field, which we regard as a subfield of both $k$ and $\bbC$.
Let $\fs(\bbbQ)$ be the $\bbbQ$-span of the Chevalley basis $\set{e_\al,h_i}$ in $\fs$,
and let $\fs(k) = \fs(\bbbQ)\ot_\bbbQ k$,
$\fs(\bbC) = \fs(\bbbQ)\ot_\bbbQ \bbC$.
We identify $\fs(k)$ with $\fs$,
regard $\fs(\bar \bbQ)$ as a $\bar \bbQ$-subalgebra of both $\fs$ and
$\fs(\bbC)$, and regard
$\Aut_{\bar\bbQ}(\fs(\bar \bbQ))$ as a subgroup of both
$\Aut_{k}(\fs)$ and $\Aut_{\bbC}(\fs(\bbC))$.  Now $\tau$ has finite order
and acts diagonally with respect to the Chevalley basis  $\set{e_\al,h_i}$, so
$\tau$ is in $\Aut_{\bar\bbQ}(\fs(\bar \bbQ))$.  On the other hand,
$\hsg$ is certainly in $\Aut_{\bar\bbQ}(\fs(\bar \bbQ))$, and therefore so is $\sg$.
But $\sg$ and $\hsg$ are conjugate in $\Aut_k(\fs)$ if and only if they are conjugate
$\Aut_{\bar\bbQ}(\fs(\bar \bbQ))$ (since conjugacy is checked by solving a system of polynomial equations
over $\bar \bbQ$),
which in turn holds if and only if they are conjugate in $\Aut_{\bbC}(\fs(\bbC))$.
Also, it is clear that $\sg$ satisfies (A1) and (A2) in $\Aut_k(\fs)$ if and only if
$\sg$ satisfies (A1) and (A2) in  $\Aut_{\bar\bbQ}(\fs(\bar \bbQ))$,  which in turn holds
if and only if $\sg$ satisfies (A1) and (A2) in $\Aut_{\bbC}(\fs(\bbC))$. So we may assume that $k=\bbC$.

It is well known
that any diagram  automorphism
satisfies (A1) and (A2) \cite[Prop.~7.9 and 7.10]{K}.
In fact,  for $0\ne \blm\in\bLm$, the
$\fg$-module $\fs^\blm$ itself satisfies condition (M). So we have
``$\Rightarrow$'' in \eqref{eq:n=1claim}.

For the converse suppose that $\sg$ satisfies (A1) and (A2).
Now,  $\tau$ fixes $\ft$ pointwise and hence $\tau$ stabilizes the root
spaces in $\fs$ relative to $\ft$. It follows that $\tau$ lies in the connected component
of the identity in the algebraic group $\Aut_k(\fs)$.  Thus,
replacing $\sg$ by a conjugate in $\Aut_k(\fs)$, we may assume that
\[\tau = \exp(2\pi ih),\]
where $h\in \mathfrak{h}^{\hat{\sigma}}$  has barycentric coordinates
satisfying certain positivity conditions \cite[Thm.~3.11 and 3.16]{GOV}.
(We will not need the definition of these coordinates nor the precise
conditions, but only the consequences that we mention below.)
Let
$\fg = \fs^\sg$ and $\hfg = \fs^\hsg$.
By (A1), $\fg$ and $\hfg$ are simple.
Also, since $\tau$ fixes $\ft$ pointwise, we see that $\ft^\sg= \ft^\hsg$, and we set
\[\fh := \ft^\sg = \ft^\hsg.\]
Then, since $\fh = \ft^\hsg$, the centralizer  $C_\fs(\fh)$ of $\fh$ in $\fs$
is a Cartan subalgebra
of $\fs$  \cite[Theorem 9]{P}.
But $\ft \subseteq C_\fs(\fh)$, so
\[\ft = C_\fs(\fh).\]
Thus, any Cartan subalgebra of $\fg$ containing $\fh$ is contained in $\ft$
and hence lies in  $\fg\cap\ft = \fh$.  So $\fh$ is a Cartan subalgebra of $\fg$, and similarly
$\fh$ is a Cartan subalgebra of $\hfg$.

We now let $\De_\fg = \De(\fg,\fh)$, $\De_\hfg = \De(\hfg,\fh)$, and
$\De = \De(\fs,\fh)$.
Then, by  Proposition~\ref{prop:N}, $\De$ is an irreducible root system in
$\fh^*$, $\De = \De_\fg$ or  $\De = (\De_\fg)_\mathrm{en}$, and
 $\De = \De_\hfg$ or  $\De = (\De_\hfg)_\mathrm{en}$.
Hence $\De_\fg = \De_\hfg$,
and consequently $\fg \simeq \hfg$.

Let
$m =\order{\sg}$ and $\hm = \order{\hsg}$.
Then $\hsg|_\fh = \sg|_\fh$ and $\hm$
is the order of $\hsg|_\fh$.  Thus,
$\hm$ divides $m$.
In fact, we claim that
\[\hm = m.\]
To see this it is  enough to show
that $\tau^\hm = 1$, or equivalently that $\tau^\hm|_{\fs_\al} = 1$ for $\al\in \De$.
Hence,  since  $\De \subseteq (\De_\fg)_\mathrm{en}$, it suffices to
to show that $\tau^\hm|_{\fs_\al} = 1$ for $\al\in (\De_\fg)_\mathrm{en}$.
So, since $\tau = \exp(2\pi ih)$,
it is enough to show that $\tau^\hm|_{\fg_\al} = 1$ for $\al\in \De_\fg$.
But if $\al\in \De_\fg$ and $x\in \fg_\al$,
$\tau^\hm (x) = \sg^{\hm}(x) = x$, proving the claim.

If $\hm = m =1$, we have $\sg = 1$.  So we can assume that
$\hm = m = 2$ or $3$.

Now let $X_r$ be the type of the simple Lie algebra $\fs$,
let $\tilde \Pi^{(\hsg)} = \set{\beta_0,\beta_1,\dots,\beta_N}$ be
the set of nodes of an affine Dynkin diagram of type $X_n^{(\hm)}$,
and let $n_0',\dots,n_N'$ be the standard labels on the diagram
$X_n^{(\hm)}$ as in  \cite[Table 3, p.~228]{GOV}.  (These are the same
as the labels in \cite[Tables Aff 2 and 3]{K} except for
type $A_{2\ell}^{(2)}$ where the labels are reversed.)
Then,  we can identify the sub-diagram
$\set{\beta_1,\dots,\beta_N}$ with the Dynkin diagram of $\De_\fg$
(relative to some base) \cite[\S 3.9]{GOV}.
With this notation, the barycentric coordinates $x_0,\dots,x_N$ of $h$ satisfy
$x_i = \frac {p_i}{m}$, where $p_0,\dots,p_N$ are nonnegative integers satisfying
$\sum_{i=0}^N n'_i p_i = \frac m \hm$ \cite[\S 3.11]{GOV}.

Since $m = \hm$, it follows that exactly one $p_i$, say
$p_a$, is nonzero.  Further, $p_a = 1$ and
\begin{equation}
\label{eq:nulla}
n'_a = 1.
\end{equation}
Also, since $\fs^\sg \simeq \fs^\hsg$,
\begin{equation}
\label{eq:nullb}
\text{the diagrams $\tilde \Pi^{(\hsg)} \setminus \set{\beta_a}$ and $\tilde \Pi^{(\hsg)} \setminus \set{\beta_0}$
are isomorphic}
\end{equation}
[ibid].  An examination of the Dynkin diagrams in
\cite[Table 3, p.~228]{GOV}  demonstrates that conditions \eqref{eq:nulla} and
\eqref{eq:nullb} imply that $\beta_a$ is conjugate
to $\beta_0$ under $\Aut(\tilde \Pi^{(\hsg)})$.  This together with $p_a = 1$ implies
that $\sg$ is conjugate to  $\hsg$ \cite[\S 3.11]{GOV}.
\end{proof}

\begin{remark}
\label{rem:Kac}
If  $\cE$ is an affine Kac-Moody Lie algebra, it is known that
the quotient $\cE^{(1)}/Z(\cE^{(1)})$ has the structure of a centreless Lie $\bbZ$-torus,
where $\cE^{(1)}$ is the derived algebra of $\cE$.
Moreover, this Lie torus is always fgc.
Hence, by Theorem \ref{thm:realization} and
Proposition \ref{prop:n=1},
$\cE^{(1)}/Z(\cE^{(1)})$ is bi-isomorphic to
a (multi)loop Lie torus $\LT(\fs,\sg,\fh)$, where
$\fs$ is a finite dimensional simple Lie algebra,
$\sg$ is a diagram automorphism of $\fs$
and $\fh$ is a Cartan subalgebra of $\fs^\sg$.
This is the classical realization theorem  \cite[Theorems 7.4 and 8.3]{K} for
affine Kac-Moody Lie algebras.
However, we have not given a new proof of the classical theorem since the fact
that $\cE^{(1)}/Z(\cE^{(1)})$ is fgc does not seem clear without using  that same theorem.
\end{remark}

\section{Bi-isomorphism and isotopy of multiloop Lie tori}
\label{sec:loopbiandiso}

In this  section, we give  necessary and  sufficient conditions for two multiloop Lie tori
$\cL$ and $\cL'$ be bi-isomorphic.  We then use that result to give
necessary and  sufficient conditions for $\cL$ and $\cL'$ to be isotopic.

\subsection{Bi-isomorphism of multiloop Lie tori}
\label{subsec:loopbi}

\begin{theorem}
\label{thm:loopbi}
Suppose that
$k$ is an algebraically closed field of characteristic 0 and $n$ is
a positive integer.  Let $\cL = \LT(\fs,\bsg,\fh)$ and $\cL'=\LT(\fs',\bsg',\fh')$
be multiloop Lie $\Zn$-tori.
Then  $\cL$ and $\cL'$ are bi-isomorphic if and only if there exists a matrix
$P=(p_{ij}) \in\GL_n(\bbZ)$ and an algebra isomorphism $\ph : \fs\to \fs'$ such that
\[\bsg' = \ph \bsg^P \ph^{-1}.\]
Moreover,  in that case, $\ph$ can be chosen with the additional property
that $\ph(\fh) = \fh'$.
\end{theorem}

\begin{proof}
The first statement follows from Proposition \ref{prop:torusequiv} and
Theorem \ref{thm:isomult}.  For the second statement, suppose that we have $P$ and $\ph$
as in the first statement.  Let $\fg = \fs^\bsg$
and $\fg' = {\fs'}^{\bsg'}$.  Then
$\ph(\fg) = \fg'$.  Thus, $\ph(\fh)$ and $\fh'$ are Cartan subalgebras
of $\fg'$.  But then there exists an automorphism $\psi_{\fg'}\in\Aut_k(\fg')$
of the form $\psi_{\fg'} = \exp(\ad_{\fg'}(x_1)) \dots \exp(\ad_{\fg'}(x_\ell))$,
with $x_1,\dots,x_\ell$ ad-nilpotent elements of $\fg'$,
such that $\psi_{\fg'}(\phi(\fh)) = \fh'$ (see the proof of Proposition
\ref{prop:torusequiv}).
Now $\psi_{\fg'}$ extends to an automorphism of $\fs'$ that
commutes with the automorphisms $\sg'_1,\dots,\sg_n'$.  Replacing
$\ph$ by $\psi\circ\ph$ we may assume  that~$\ph(\fh) = \fh'$.
\end{proof}

\begin{remark}
\label{rem:caution}   Suppose that $\LT(\fs,\bsg,\fh)$ is a multiloop Lie $\Zn$-torus
(so $\bsg\in \cfo_n(\fs)$ satisfies (A1), (A2) and (A3)), and let $P\in \GL_n(\bbZ)$.  Then, it is easy to see
that $\bsg^P$ also satisfies (A1) and (A2).  However, the reader is cautioned that
it is not true in general that $\bsg^P$ satisfies (A3).  (An example is easily found
when $n=2$, $\order{\sg_1} = 2$ and  $\order{\sg_2} = 1$.)  In fact,
any orbit of $\GL_n(\bbZ)$ acting on $\cfo_n(\fs)$ contains
an element $\bsg$ that satisfies (A3) (see Proposition ~\ref{prop:orbitrep} below).
This latter fact will be important in our discussion of classification in section \ref{sec:approach}
 (see Theorem \ref{thm:class}(iii) below).
\end{remark}

\begin{remark}
Suppose  that $n=1$.
Since any diagram automorphism of $\fs$ is conjugate to its inverse
in $\Aut_k(\fs)$, it follows from Theorem \ref{thm:loopbi} and Proposition
\ref{prop:n=1} that two (multi)loop Lie $\bbZ$-tori $\LT(\fs,\sg,\fh)$ and $\LT(\fs',\sg',\fh')$
are bi-isomorphic if and only if there exists
an algebra isomorphism $\ph: \fs \to \fs'$ such that
$\sg' = \ph \sg \ph^{-1}$.
\end{remark}

\begin{corollary}
\label{cor:loopbi1}
Suppose that $\cL$ and $\cL'$ are  as in Theorem \ref{thm:loopbi}.
If $\cL$ and $\cL'$ are bi-isomorphic, then there exists an
algebra isomorphism $\ph : \fs\to \fs'$ such that
$\langle \bsg'\rangle = \ph \langle \bsg \rangle \ph^{-1}$.
\end{corollary}

\begin{proof}  This follows from Theorem \ref{thm:loopbi} and
\eqref{eq:bsggen}.
\end{proof}

\begin{corollary}
\label{cor:loopbi2}  Two
untwisted multiloop Lie  $\Zn$-tori
$\LT(\fs,\boldone,\fh)$ and $\LT(\fs',\boldone,\fh')$
are bi-isomorphic if and only if $\fs$ and $\fs'$ are  isomorphic,
\end{corollary}

\begin{corollary}
\label{cor:untwisted}
Suppose that $\cL = LT(\fs,\bsg,\fh)$ is a multiloop Lie $\Zn$-torus.
Then the following are equivalent:
\begin{itemize}
\item[(a)] $\bsg = \boldone$.
\item[(b)] $\cL$ equals the untwisted multiloop Lie $\Zn$-torus $\cL = LT(\fs,\boldone,\fh)$.
\item[(c)] $\cL$ is bi-isomorphic to an untwisted multiloop Lie  $\Zn$-torus $LT(\fs',\boldone,\fh')$.
\item[(d)] $\sg_1,\dots,\sg_n$ are contained in a common torus of the algebraic group $\Aut_k(\fs)$
(using the terminology of  \cite[\S~16.2]{H}  for example).
\item[(e)] $\fh$ is a Cartan subalgebra of $\fs$.
\end{itemize}
\end{corollary}

\begin{proof} We use the notation $\fg$, $\De_\fg$, $\De$ and $Q$
from Proposition \ref{prop:N}.  Now
``(a)$\Rightarrow$(b)'' and ``(b)$\Rightarrow$(c)'' are trivial.  Also,
``(c)$\Rightarrow$(a)'' follows from Theorem \ref{thm:loopbi}, hence
``(c)$\Rightarrow$(d)''.  To prove ``(d)$\Rightarrow$(e)'',
suppose that $\sg_1,\dots,\sg_n$ are contained in a torus $T$ of $\Aut_k(\fs)$.
Let $\Lie(T)$ be the Lie algebra of $T$ in $\Lie(\Aut_k(\fs)) = \ad_\fs(\fs)$.
Then  $\Lie(T) = \ad_\fs(\ft)$, where $\ft$ is a Cartan subalgebra of $\fs$.
But $\ft$ is fixed pointwise by $T$ and hence by each $\sg_i$.
Thus, $\ft \subseteq \fg = \fs^{\bar 0}$, so $\ft$ is a Cartan subalgebra of $\fg$.
Hence the simple Lie algebras $\fg$ and $\fs$ have the same rank, so we have (e).
Finally, to prove ``(e)$\Rightarrow$(a)'', suppose that $\fh$ is a Cartan subalgebra of $\fs$.
Then $\De$ is reduced and therefore, by \eqref{eq:propN3}, $\De = \De_\fg$.
Next, if $\al\in \Dec = \De_\fg^\times$, we have $0\ne \fg_\al \subseteq \fs_\al$, so,
$\fg_\al = \fs_\al$ since $\dim(\fs_\al)=1$.  Thus, $\fs \subseteq \fg$, and hence $\fs = \fg$,
which implies (a).
\end{proof}

\begin{remark}  Suppose that $\fs$ is a finite dimensional simple Lie algebra,
$\bsg = (\sg_1,\dots,\sg_n)\in \cfo_n(\fs)$   and
$\boldm  \in \bbZ^n_+$ with $\bsg^\boldm = \boldone$.
If the automorphisms  $\sg_1,\dots,\sg_n$ are contained
in a common torus of $\Aut_k(\fs)$, it is not difficult to show using techniques of Galois cohomology
that the multiloop algebra $M_\boldm(\fs,\bsg)$ is isomorphic as a Lie algebra to
the untwisted multiloop algebra $\fs\ot k[z_1^{\pm 1},\dots,z_n^{\pm 1}]$.  We have not used
that general fact in the graded result  Corollary \ref{cor:untwisted}.
\end{remark}

\subsection{Isotopy of multiloop Lie tori}
\label{subsec:loopisotopy}

Throughout this  subsection suppose that $\fs$ is a finite dimensional simple Lie algebra over
$k$, $\bsg = (\sg_1,\dots,\sg_n)$ is a sequence  in  $\cfo_n(\fs)$ satisfying (A1)--(A3), and
$\fh$ is a Cartan subalgebra of $\fg = \fs^\bsg$.
Let $\boldm = (m_1,\dots,m_n)$, where $m_i = \order{\sg_i}$,
$\De = \De(\fs,\fh)$ and $Q = \spann_\bbZ(\De)$; and let
$\fs = \oplus_{\al\in Q}\fs_\al$ be the root space
decomposition of $\fs$ relative to the adjoint action of $\fh$.
Finally,  let
\[\cL = \LT(\fs,\bsg,\fh)\]
be the multiloop Lie $\Zn$-torus determined by $\fs$, $\bsg$ and $\fh$.

We will be considering isotopes of $\cL$.  (See subsection
\ref{subsec:isotopes} for the terminology and notation used here.)
To do this we use
elements $\shom\in \Hom(Q,\Zn)$, which we write as
$\shom = (\shom_1,\dots,\shom_n)$, where $\shom_i\in \Hom(Q,\bbZ)$ for $1\le i \le n$.

\begin{proposition}
\label{prop:multisotope}
Suppose that $\shom= (\shom_1,\dots,\shom_n)\in \Hom(Q,\Zn)$
is admissible for~$\cL$.  For $1\le i \le n$,
define $\tau_i$ and $\tsg_i$ in $\Aut_k(\fs)$ by
\[\tau_i(x_\al) = \zeta_{m_i}^{-\shom_i(\al)} x_\al\]
for $\al\in Q$, $x_\al\in \fs_\al$,
and
\[\tsg_i = \tau_i\sg_i = \sg_i\tau_i\in \Aut_k(\fs)\]
for $1\le i \le n$.
Then $\tbsg = (\tsg_1,\dots,\tsg_n)\in \cfo_n(\fs)$, $\tbsg$  satisfies
(A1)--(A3), $\fh$ is a Cartan subalgebra
of $\fs^\tbsg$, $\order{\tsg_i} = m_i$ for $1\le i \le n$, and
$\cLs$ is $Q\times \Zn$-graded isomorphic to
$\LT(\fs,\tbsg,\fh)$.
\end{proposition}

\begin{proof} It is clear that the automorphisms $\sg_1,\dots,\sg_n,\tau_1,\dots,\tau_n$
commute.  So our definition of $\tsg_i$ makes sense, the automorphisms
$\tsg_1,\dots,\tsg_n$ commute, $\tsg_i^{m_i} = 1$ and $\fh$ is fixed
pointwise by each $\tsg_i$.  So, using Construction \ref{con:LTconstruction},
we can form $\tL = \LT_\boldm(\fs,\fh,\tbsg)$.

Now we have
\[ \cL = \textstyle \sum_{(\al,\lm) \in Q\times \Zn} \fs_\al^{\blm}\ot z^\lm
\]
where $\fs$ is $Q\times \bLm$-graded with the $Q$-grading determined by $\fh$ and the $\bLm$-grading
determined by $\bsg$.  In contrast, we have
\[\tL = \textstyle \sum_{(\al,\lm) \in Q\times \Zn} \tilde \fs_\al^{\blm}\ot z^\lm,\]
where we use the notation $\tilde\fs$ for the
$Q\times \bLm$-graded algebra $\fs$ with the  $Q$-grading determined by $\fh$ and the
$\bLm$-grading determined by $\tbsg$.  One checks directly using the definitions
that
\[\tilde\fs_\al^\blm = \fs_\al^{\overline{\lm + \shom(\al)}}.\]

Note that
$\cL$, $\cLs$ and $\tL$ are  all subalgebras of
$\fs \ot k[z_1^{\pm 1},\dots,z_n^{\pm 1}]$.
We define a $k$-algebra automorphism $\psi$ of  $\fs\otimes k[z_1^{\pm 1},\dots,z_n^{\pm 1}]$ by
\[\psi(x_\al \ot a) = x_\al \ot az^{-\shom(\al)}\]
for $x_\al\in \fs_\al$, $\al\in Q$, $a\in k[z_1^{\pm 1},\dots,z_n^{\pm 1}]$.
Then, for $\al\in Q$, $\lm\in\Zn$, we have
\[
\psi((\cLs)_\al^\lm) =  \psi(\cL_\al^{\lm+\shom(\al)})
= \psi(\fs_\al^{\overline{\lm+\shom(\al)}}\ot z^{\lm+\shom(\al)})
= \fs_\al^{\overline{\lm+\shom(\al)}}\ot z^{\lm}
= \tilde\fs_\al^{\blm}\ot z^{\lm}
= \tL_\al^\lm.\]
So $\psi$ restricts to a $Q\times\Lm$-graded isomorphism of $\cLs$ onto $\tL$.

Next, by Proposition \ref{prop:N}, $\cL$ is a Lie torus.
Hence, since $\shom$ is admissible for $\cL$,
$\cLs$ is a Lie torus and therefore so is $\tL$.  Thus, by Proposition \ref{prop:N},
$\tbsg$ satisfies (A1)--(A3), $\fh$ is a Cartan subalgebra
of $\fs^\tbsg$, and $\order{\tsg_i} = m_i$ for $1\le i \le n$.
\end{proof}

\begin{definition}
\label{def:Ad}
As in \cite[\S 1.10]{KW},   we define
$\Ad : \Hom(Q,k^\times) \to \Aut_k(\fs)$ by
\[\Ad(\rho)(x_\al) =  \rho(\al)x_\al\]
for $x_\al\in\fs_\al$, $\al\in Q$.  Then $\Ad$ is a group homomorphism, so
$\Ad( \Hom(Q,k^\times))$
is a subgroup  of $\Aut_k(\fs)$.
\end{definition}

\begin{remark}
\label{rem:tangent}
$\Aut_k(\fs)$  is an algebraic group with Lie algebra
$\ad_\fs(\fs)$, and $\Ad$ is a homomorphism of
algebraic groups, so $H = \Ad(\Hom(Q,k^\times))$ is a closed connected
subgroup of $\Aut_k(\fs)$.  In fact, it is not difficult to show that
$H$ is the torus in $\Aut_k(\fs)$ whose Lie algebra is
$\ad_\fs(\fh)$. We omit the proof of this fact, since we won't use it.
\end{remark}

\begin{lemma}
\label{lem:isotopy} Let $H = \Ad( \Hom(Q,k^\times))$ and  let
$\Pi$ be a base for $\De$.   Suppose $\btau = (\tau_1,\dots,\tau_n)\in H^n$
with $\btau^\boldm = \boldone$, and let $\tbsg = (\tsg_1,\dots,\tsg_n)\in \cfo_n(\fs)$,
where $\tsg_i = \tau_i\sg_i$ for $1\le i \le n$.
Then the following are equivalent:
\begin{itemize}
\item[(a)] $\tbsg$ satisfies (A1)--(A3).
\item[(b)] $\tbsg^P$ satisfies (A1)--(A3) for some $P\in\GL_n(\bbZ)$.
\item[(c)]$\fs^\tbsg \cap \fs_\al \ne 0$ for $\al\in \De$.
\item[(d)]$\fs^\tbsg \cap \fs_\al \ne 0$ for $\al\in \Pi$.
\end{itemize}
Further, in that case, $\fh$ is a Cartan subalgebra
of $\fs^\tbsg$ and $\cL$ is isotopic to $\LT(\fs,\tbsg,\fh)$.
\end{lemma}

\begin{proof}  ``(a) $\Rightarrow$ (b)'' is trivial.

``(b) $\Rightarrow$ (c)'' Let $\bsg' = (\sg_1',\dots,\sg_n') = \tbsg^P$, so by assumption $\bsg'$ satisfies
(A1)--(A3).  Now, by \eqref{eq:bsggen}, we have $\fs^{\bsg'}= \fs^{\tbsg}$, and we let
\[\fg = \fs^\bsg
\andd
\fg' = \fs^{\bsg'} = \fs^\tbsg.
\]
Since $\bsg$ and $\bsg'$ satisfy (A1), $\fg$ and $\fg'$ are simple.

Next $\fh$ is a Cartan subalgebra of $\fg$ and $\fh \subseteq \fs^\tbsg = \fg'$.
So $\fh$ is contained in a Cartan subalgebra $\fh'$ of $\fg'$.
Let $\pi : {\fh'}^* \to \fh^*$ be the restriction map.  Each root space $\fs_\al$ for
the adjoint action of $\fh$
is stabilized by the adjoint action of $\fh'$.  Hence,  we have
\[\fs_\al = \textstyle\sum_{\al'\in{\fh'}^*,\, \pi(\al') = \al} \fs_{\al'}\]
for $\al\in\fh^*$,
where $\fs_{\al'}$ is the $\al'$-root space of $\fs$
relative to the adjoint action of $\fh'$.  Consequently, since $\De = \De(\fs,\fh)$, we have
$\pi(\De(\fs,\fh')) = \De$, so
\begin{equation}
\label{eq:fhfh'}
\Deic \subseteq \pi(\De(\fs,\fh')_\text{ind}^\times).
\end{equation}

Now let $\al\in\Deic$.  Then, by
\eqref{eq:fhfh'}, $\al\in\pi(\al')$, where $\al'\in \De(\fs,\fh')_\mathrm{ind}^\times$.
Now, since $\bsg'$ satisfies (A1)--(A3), $\LT(\fs,\bsg',\fh')$ is a Lie torus, so
it satisfies (LT2)(i).  Thus, $\LT(\fs,\bsg',\fh')_{\al'}^0 \ne 0$, and therefore
$\fs^{\bsg'} \cap \fs_{\al'} \ne 0$ by \eqref{eq:LTgrading}. Since
$\fs_{\al'} \subseteq \fs_\al$, we have  $\fs^{\bsg'} \cap \fs_{\al} \ne 0$, so
$\fs^{\tbsg} \cap \fs_{\al} \ne 0$.

``(c) $\Rightarrow$ (d)'' is trivial.

``(d) $\Rightarrow$ (a)'' If $1\le i  \le n$, $\tau_i^{m_i} = 1$ and
$\tau_i = \Ad(\rho_i)$ for some $\rho_i\in \Hom(Q,k^\times)$.
Hence, $\rho_i(\al)^{m_i} = 1$ for $\al\in\De$ and $1\le i \le n$.  Thus,
we may choose $\shom_i\in \Hom(Q,\bbZ)$ such that $\rho_i(\al) = \zeta_{m_i}^{-\shom_i(\al)}$ for $\al\in Q$.
 (Choose $\shom_i$ so that this holds for $\al$ in $\Pi$,
in which case it holds for all $\al\in Q$.)  So
\[\tau_i(x) = \zeta_{m_i}^{-\shom_i(\al)} x\]
for $x\in \fs_\al$, $\al\in Q$.  Let $\shom = (\shom_1,\dots,\shom_n) \in \Hom(Q,\Zn)$.

To show that  $\shom$ is admissible for $\cL$, let $\al\in\Pi$ .  By assumption,
we may choose $0\ne x \in \fs^\tbsg \cap \fs_\al$.  Thus, $\sg_i\tau_i(x) = x$, so
$\sg_i(x) = \zeta_{m_i}^{\shom_i(\al)} x$. for all $i$.  Hence $x\in \fs_\al^{\shom(\al)},$
so $\cL_\al^{\shom(\al)} \ne 0$ by \eqref{eq:LTgrading}. Hence, $\shom(\al)\in \Lm_\al$,
and $\shom$ is admissible for $\cL$.

So, by Proposition \ref{prop:multisotope}, (a) holds, $\fh$ is a Cartan subalgebra
for $\fs^\tbsg$, and $\cL$ is isotopic to   $\LT(\fs,\tsg,\fh)$.
\end{proof}

\begin{theorem}
\label{thm:isotopy}
Let $\cL = \LT(\fs,\bsg,\fh)$  and $\cL' = \LT(\fs',\bsg',\fh')$ be multiloop Lie
$\Zn$-tori,  let $\De = \De(\fs,\fh)$, $Q = Q(\De)$, $m_i = \order{\sg_i}$, and let
$H = \Ad( \Hom(Q,k^\times))$.
Then, $\cL$ is isotopic to $\cL'$ if and only if there exist $\tau_1,\dots,\tau_n\in H$  with
$\tau_j^{m_j} = 1$, $P=(p_{ij}) \in\GL_n(\bbZ)$ and an algebra isomorphism $\ph : \fs\to \fs'$ such that
\[\bsg' = \ph \tilde \bsg^P \ph^{-1},\]
where $\tilde\sg = (\tau_1\sg_1,\dots,\tau_n\sg_n)$.
\end{theorem}

\begin{proof}    ``$\Rightarrow$'' Suppose that $\cL$ is isotopic to $\cL'$.  Then,
$\cLs$ is bi-isomorphic to $\cL'$ for some $\shom = (\shom_1,\dots,\shom_n)\in \Hom(Q,\Zn)$
that is admissible for $\cL$.  Let $\tau_i$ and $\tbsg$  be defined as in Proposition
\ref{prop:multisotope}.  Then $\tau_i = \Ad(\rho_i)$, where $\rho_i\in \Hom(Q,k^\times)$
is defined by  $\rho_i(\al) = \zeta_{m_i}^{-\shom_i(\al)}$.  So $\tau_i\in H$ and
$\tau_i^{m_i} = 1$ for all $i$.  But, by Proposition \ref{prop:multisotope},
$\cLs$ is $Q\times\Lm$-graded isomorphic to $\tL = \LT(\fs,\tbsg,\fh)$.  Hence,
$\tL$ is bi-isomorphic to $\cL'$, and our conclusion follows from Theorem \ref{thm:loopbi}.

``$\Leftarrow$'' Since  $\bsg'$ satisfies (A1)--(A3), $\tbsg^P = \ph^{-1}\bsg'\ph$
satisfies (A1)--(A3).  So, by Lemma \ref{lem:isotopy}, $\tbsg$ satisfies (A1)--(A3),
$\fh$ is a Cartan subalgebra of $\fs^\tbsg$, and
$\cL$ is isotopic to $\LT(\fs,\tbsg,\fh)$.
But, by Theorem \ref{thm:loopbi}, $\LT(\fs,\tbsg,\fh)$ is bi-isomorphic to
$\cL'$.  So $\cL$ is isotopic to  $\cL'$.
\end{proof}

\begin{corollary}
\label{cor:isotopy} If a multiloop Lie $\Zn$-torus $\LT(\fs,\bsg,\fh)$
is isotopic to an untwisted multiloop Lie $\Zn$-torus, then $\bsg = \boldone$.
Two untwisted multiloop Lie $\Zn$-tori $\LT(\fs,\boldone,\fh)$ and
$\LT(\fs',\boldone,\fh')$ are isotopic if and only if $\fs$
and $\fs'$ are isomorphic.
\end{corollary}

\begin{proof}  The second statement follows immediately from Theorem \ref{thm:isotopy}.
For the first statement, suppose that $\LT(\fs,\bsg,\fh)$ is isotopic to
$\LT(\fs',\boldone,\fh')$.  Choose $\tau_1,\dots,\tau_n$, $P$,
$\tbsg$ and $\ph$ as in Theorem \ref{thm:isotopy}.  Then,  $\tbsg^P = \ph^{-1} \boldone \ph = \boldone$,
so $\tbsg = 1$.  Thus, $\sg_1,\dots,\sg_n$ are contained in the torus $H$,
so, by Corollary \ref{cor:untwisted},
$\bsg = \boldone$.
\end{proof}

\begin{corollary}
\label{cor:isotopy1}  Two  (multi)loop Lie $\bbZ$-tori $\LT(\fs,\sg,\fh)$ and
$\LT(\fs',\sg',\fh')$ are isotopic if and only they are bi-isomorphic.
\end{corollary}

\begin{proof}  Fix a Cartan subalgebra $\ft$ of $\fs$ and an \'epinglage $(\set{\al_i},\set{e_i})$
of $(\fs,\ft)$.  Let $D$ denote the Dynkin diagram of $\fs$ with respect to
$\ft$ and $\set{\al_i}$.  Then $\Aut(D)$ can be identified
with the group of diagram automorphisms of $\fs$ with respect to
$\ft$ and $(\set{\al_i},\set{e_i})$, and we
have  $\Aut(\fs) = \Aut(\fs)^0 \rtimes \Aut(D)$ \cite[chap.~VIII, \S~5]{B2}.

Now one direction in the Corollary is trivial, so we assume that $\LT(\fs,\sg,\fh)$ and
$\LT(\fs',\sg',\fh')$ are isotopic.   By Proposition \ref{prop:n=1} and Theorem \ref{thm:loopbi},
we can assume $\fs = \fs'$ and $\sg,\sg'\in \Aut(D)$.
Further, by Theorem \ref{thm:isotopy}, we have
$\sg' = \ph \tau \sg^{\pm 1} \ph^{-1}$, where $\ph\in \Aut_k(\fs)$ and
$\tau \in \Aut(\fs)^0$.  Projecting this equation onto
$\Aut(D)$ shows that $\sg'$ is conjugate to $\sg^{\pm 1}$ in $\Aut(D)$
and hence also in $\Aut_k(\fs)$.  Thus, $\LT(\fs,\sg,\fh)$ and
$\LT(\fs',\sg',\fh')$ are bi-isomorphic by Theorem \ref{thm:loopbi}.
\end{proof}

\subsection{An example}
We expect that  Corollary \ref{cor:isotopy1} is also true
for multiloop Lie $\bbZ^2$-tori.  However, it is not true
for multiloop Lie $\bbZ^3$-tori, as the next example  shows.

\begin{example}
\label{ex:B3}
Let  $f: k^7 \times k^7 \to k$ be the symmetric bilinear form with matrix
$\diag(J_3,\Id_4)$ relative to the standard basis,
where  $J_3 = \begin{bmatrix} 0&0&1\\0&1&0\\1&0&0 \end{bmatrix}$
and $\Id_4$ is the $4\times 4$-identity matrix.  Let $\fs = \mathfrak{0}(f)$ be the
orthogonal Lie algebra of endomorphisms of $k^7$ that are skew relative to $f$,
in which case $\fs$ is simple of type B$_3$.  Writing endomorphisms of $k^7$ as
matrices relative to the standard basis, we see that $\fs$ is the Lie algebra
of all matrices of the form
\[\begin{bmatrix} A &-J_3C^t \\ C & B\end{bmatrix}.\]
where $A\in \Mat_3(k)$, $B\in \Mat_4(k)$, $C \in \Mat_{4\times 3}(k)$, $A^t = -J_3AJ_3$
and $B^t = -B$.  Also, if $\SO(f)$ denotes the special orthogonal group of $f$,
the map $c\mapsto c_g$ is an isomorphism of $\SO(f)$ onto $\Aut_k(\fs)$, where $c_g(x) = gxg^{-1}$
for $x\in \fs$ \cite[Theorem IX.6]{J}.

Next let $\bsg = (\sg_1,\sg_2,\sg_3)= (c_{d_1},c_{d_2},c_{d_3})\in \cfo_3(\fs)$,  where
\[d_1 = \diag(\Id_3,-1,1,1,-1),\ d_2 = \diag(\Id_3,1,-1,1,-1),\ d_3 = \diag(\Id_3,1,1,-1,-1)\]
in $\SO(f)$.
Then,
\[\fg = \fs^\bsg = \set{\begin{bmatrix} A &0 \\ 0 & 0\end{bmatrix}
\suchthat  A\in \Mat_3(k),\ A^t = -J_3AJ_3},\]
is simple of type A$_1$ with Cartan subalgebra $\fh = k(e_{11}-e_{33})$, and we have
$\De = \De(\fs,\fh) = \set{\ep_1,0,-\ep_1}$, where $\ep_1(e_{11}-e_{33}) = 1$.  So
$Q = Q(\De) = \bbZ \ep_1$.

Let $\boldm = (2,2,2)$ and, as in Construction \ref{con:LTconstruction}, let $\bLm = (\bbZ/(2))^3$.
Then, one checks that for $\bar 0  \ne \blm\in\bLm$, the $\fg$ module $\fs^\blm$ is isomorphic
to the adjoint module in 4 cases, whereas $\fs^\blm$
is 2-dimensional with trivial action in the other 3 cases.
So $\bsg$ satisfies (A1) and (A2), and, since $\order{\langle \bsg \rangle} = 8$,
$\bsg$ satisfies (A3).  Thus, by Proposition \ref{prop:N},
$\cL = \LT(\fs,\bsg,\fh)$ is a
multiloop Lie $\bbZ^3$-torus of type $\De$.

Now let $H = \Ad(\Hom(Q,k^\times))$, in which case
$H = \set{c_{\diag(a,1,a^{-1},\Id_4)} \suchthat a\in k^\times}$.
Put
\[d = \diag(-1,1,-1,\Id_4)\in \SO(f), \quad \tau_1=\tau_2=\tau_3 = c_d,\]
$\btau = (\tau_1,\tau_2,\tau_3)$ and $\tbsg = (\tsg_1,\tsg_2,\tsg_3)$ with
$\tsg_i = \tau_i\sg_i$ for $i=1,2,3$.
Then, one checks that $e_{71}-e_{37}\in \fs^\tbsg\cap \fs_{\ep_1}$.
So by Lemma \ref{lem:isotopy}, $\tbsg$ satisfies (A1)--(A3), $\fh$ is a Cartan
subalgebra of $\fs^\tbsg$, and $\cL$ is isotopic to $\tilde\cL = \LT(\fs,\tbsg,\fh)$.

Note finally that $\langle \bsg\rangle = c_D$ and $\langle \tbsg\rangle = c_{\tilde D}$,
where $D = \langle d_1,d_2,d_3\rangle$ and $\tilde D = \langle dd_1,dd_2,dd_3\rangle$.
Then, $\tilde D$ contains an element $(dd_1)(dd_2)(dd_3)$
with $(-1)$-eigenspace
of dimension $6$, whereas $D$ clearly contains no such element.  Thus, $D$ and $\tilde D$
are not conjugate in $\SO(f)$, so $\langle \bsg\rangle$ and $\langle \tbsg \rangle$
are not conjugate in $\Aut_k(\fs)$.  Hence, by Corollary \ref{cor:loopbi1}, $\cL$ and
$\tilde \cL$ are not bi-isomorphic.
\end{example}

\begin{remark}
\label{rem:B3} Suppose that we have the assumptions and notation of Example \ref{ex:B3}
with just one change:  we let $\tau_1 = \tau_2 = c_d$ and $\tau_3 = \id$.  In this case
$\fs^\tbsg\cap \fs_{\ep_1} = 0$, so, by Lemma \ref{lem:isotopy},
$\tbsg$ does not satisfy (A1)--(A3).  In fact, $\fs^\tbsg = k (e_{11}-e_{33})$, which
is not simple.
\end{remark}

\begin{remark}
\label{rem:B3coord}
Example \ref{ex:B3} is presented in a very different way in \cite[Example 8.9]{AF},
where $\cL$ and $\cL'$ are constructed using the Tits-Kantor-Koecher Lie algebra
construction from spin-factor Jordan tori $\cJ$ and $\cJ'$ that are isotopic
but not isograded-isomorphic. (See also Remark \ref{rem:motivation2} above.)
\end{remark}

\section{An approach to classification}
\label{sec:approach}

In this section we describe an approach to  the classification of fgc Lie tori
first up to bi-isomorphism and then up to isotopy.

\subsection{A proposition about finite abelian groups}
\label{subsec:finiteabelian}

In  this subsection,  we prove a  proposition about finite abelian groups.  (This will
be applied in subsection \ref{subsec:approach} to the group generated by a sequence of commuting finite order
automorphisms.)  We will prove this proposition using the following
lemma about   integral matrices.

\begin{lemma}
\label{lem:matrix}
Suppose that $m_1,\ldots,m_n$ are positive integers and
$m_{i+1}$divides $m_i$ for $1\leq i<n$. Let $\mathcal{M}$ be the right
ideal generated by $\diag(m_1,\ldots,m_n)$ in $\Mat_n(\bbZ)$, the
ring of $n\times n$ matrices over $\bbZ$. If $A,B\in \Mat_n(\bbZ)$ with
\[
AB\equiv \Id_n\pmod{\cM},
\]
then there is $P\in \GL_n(\bbZ)$ with
$AP\equiv \diag(1,\ldots,1,p) \pmod{\cM}$ where
\begin{equation}
\label{eq:pcond}
0\leq p\leq\left\lfloor\frac{m_n}{2}\right\rfloor  \andd \gcd(p,m_n)=1
\end{equation}
(so $p=0$, if $m_n=1$). Moreover, $p$ is uniquely determined by the
equivalence class of $A \pmod{\cM}$.
\end{lemma}

\begin{proof}
We first note that
$(a_{ij})\equiv(a_{ij}^{\prime}) \pmod{\cM}$ if and only if
$a_{ij}\equiv a_{ij}^{\prime}\pmod{m_i}$
for all $i,j$. Since $m_n$ divides $m_i$, we see that
$\det(A)\pmod{m_n}$ is an invariant of the equivalence class of
$A \pmod{\cM}$. Now if
$AP\equiv \diag(1,\ldots,1,p)\pmod{\cM}$ and $P\in \GL_n(\bbZ)$, then
$p\equiv\det(AP)\equiv\pm\det(A)\pmod{m_n}$ and $0\leq
p\leq\left\lfloor\frac{m_n}{2}\right\rfloor$ uniquely determine $p$.

To show the first statement, we shall repeatedly adjust $A$ by replacing $A$
by some $A^{\prime}\equiv AP \pmod{\cM}$ and $B$ by
$P^{-1}B$ for some $P\in \GL_n(\bbZ)$ until $A$
has the form $\diag(1,\ldots,1,p)$, where
$p$ satisfies  \eqref{eq:pcond}.
We suppose $A=(a_{ij})$ is of the form
$\left[
\begin{array}
[c]{cc}
\Id_{k-1} & 0\\
\ast & \ast
\end{array}
\right]  $ and use induction on $k$. We can replace $a_{kk}$ by
$a_{kk}+m_k$ to assume $a_{kk}\neq0$. Let $d$ be the positive $\gcd$ of
$a_{kk},\ldots,a_{kn}$. Since elementary column operations are achieved by
right multiplying by an element of $\GL_n(\bbZ)$, we can adjust columns
$k$ to $n$ to assume $a_{kk}=d$ and $a_{kj}=0$ for $k<j\leq n$. Moreover,
if $k<n$, we can then take  $a_{k,n}=m_k$, get a new $\gcd$, and again
adjust columns $k$ to $n$ to assume $d\mid m_k$. Let $B=(b_{ij})$. The
first $k$ entries in the $k$th column of
$AB\equiv \Id_n\pmod{\cM}$ give
$b_{ik}\equiv 0\pmod{m_i}$ for $1\leq i<k$ and
\[
a_{k1}b_{1k}+\ldots+a_{k,k-1}b_{k-1,k}+db_{kk}\equiv 1\pmod{m_k}.
\]
Since $m_k$ divides $m_i$ for $i<k$, we have
$b_{ik}\equiv0\pmod{m_k}$. Thus,
$db_{kk}\equiv 1\pmod{m_k}$;
i.e., $\gcd(d,m_k)=1$. If $k<n$, then $d\mid m_k$ so $d=1$.
In general, since $db_{kk}\equiv 1 \pmod{m_k}$, we can add
multiples of the $k$th column to the previous columns to get
$a_{kj}\equiv 0 \pmod{m_k}$ for $1\leq j<k$; i.e., we can adjust $A$ to assume
that $a_{kk}=d$ and $a_{kj}=0$ for $j\neq k$. If $k<n$, then $d=1$ and we
have completed the induction step. If $k=n$, we can adjust $A$ by
multiplying the $n$th column by $-1$, if necessary, to get
$d\equiv p \pmod{m_n}$ with
$0\leq p\leq\left\lfloor\frac{m_n}{2}\right\rfloor$.
\end{proof}

\begin{remark}
\label{rem:actionGL} Suppose that $G$ is  a finite abelian group.
Since $G$ is finite and abelian, we have $\cfo_n(G) = G^n$, so we have
the right action $(\bsg,P) \mapsto \bsg^P$ of
$\GL_n(\bbZ)$  on $G^n$ (see Notation \ref{not:action}(i)).
Next let
\[
\mg(G):=\text{ minimum number of generators of }G.
\]
It is well-known that $\mg(G)$ is also equal to the length of the sequence of
invariant factors of $G$.
Finally, let
\[
\gs_n(G):=\{(\sg_1,\ldots,\sg_n)\in G^{n}\mid G=\left\langle
\sg_1,\ldots,\sg_n\right\rangle \}.
\]
If $\mg(G)\leq n$, then $\gs_n(G)\neq\emptyset$ and, by \eqref{eq:bsggen},
the right action of
$\GL_n(\bbZ)$ on $G^n$ restricts to a right action of $\GL_n(\bbZ)$
on $\gs_n(G)$.
\end{remark}

We now determine representatives of the  orbits of the action of $\GL_n(\bbZ)$
on $\gs_n(G)$.

\begin{proposition}
\label{prop:orbitrep}Suppose that $G$ is a finite abelian group and
$r:=\mg(G)\leq n$. Choose  $\btau = (\tau_1,\ldots,\tau_n) \in G^n$ such that
$G=\left\langle \tau_1\right\rangle \times\cdots\times\left\langle \tau
_n\right\rangle $ (as an internal direct product) and $\left\vert \tau
_{i+1}\right\vert $ divides $\left\vert \tau_{i}\right\vert $ for $1\leq i<n$.
(So the decreasing sequence of invariant factors of $G$ is
$\order{\tau_1},\ldots,\left\vert \tau_r\right\vert $, and we have
$\left\vert \tau_{i}\right\vert =1$ for $r<i\leq n$.) Then
\[
(\tau_1,\ldots,\tau_{n-1},\tau_n^{p}),
\]
where
\begin{equation}
\label{eq:pcond2}
0\leq p\leq
\left\lfloor\frac{\order{\tau_n}}{2}\right\rfloor
\andd
\gcd(p,\order{\tau_n}) =1,
\end{equation}
is a nonredundant list of representatives of the orbits of
$\GL_n(\bbZ)$ acting on $\gs_n(G)$,
(If $r<n$ then $\order{\tau_n} = 1$ and  the only choice for $p$ is $p=0$.). Consequently, every orbit contains an element
$(\sg_1,\ldots,\sg_n)$ with the property that $\left\vert
G\right\vert =\left\vert \sg_1\right\vert \cdots\left\vert \sigma
_n\right\vert $. Also, if $r<n$ or if the smallest invariant factor of $G$
is less than or equal to $4$, then $\GL_n(\bbZ)$ acts transitively on
$\gs_n(S)$.
\end{proposition}

\begin{proof}
It suffices to prove the first statement. Every
$\bsg\in G^{n}$ can be written as
$\bsg =\btau^{A}$
for some $A\in \Mat_n(\bbZ)$.
Also,
$\btau^A= \btau^{A'}$
if and only if
$A\equiv A'\pmod{\cM}$, where $\mathcal{M}$ is the right ideal in $\Mat_n(\bbZ)$
generated by
$\diag(\order{\tau_1},\ldots,\order{\tau_n} )$.

Now suppose that  $\bsg\in\gs_n(G)$  and write
$\bsg=\btau^{A}$ where $A\in \Mat_n(\bbZ)$.
Then $\btau = \bsg^B$ for some
$B\in \Mat_n(\bbZ)$.  So
$\btau=\btau^{AB}$
and hence $AB\equiv \Id_n \pmod{\cM}$.
Then by Lemma \ref{lem:matrix},
there exists $P\in\GL_n(\bbZ)$ such that $AP \equiv \diag(1,\dots,1,p) \pmod{\cM}$
where $p$ satisfies \eqref{eq:pcond2}.  So
$\bsg^P = \btau^{AP} =  (\tau_1,\dots,\tau_n^p)$.

Finally, suppose that $(\tau_1,\dots,\tau_n^{p_1})^P = (\tau_1,\dots,\tau_n^{p_2})$
for some $P\in\GL_n(\bbZ)$, where $p_1,p_2$ satisfy \eqref{eq:pcond2}.
Then, $\diag(1,\dots,1,p_1)P \equiv \diag(1,\dots,1,p_2) \pmod{\cM}$,
so $p_1 = p_2$ by  Lemma \ref{lem:matrix}.
\end{proof}

\subsection{Absolute type}
\label{subsec:absolute}

We next recall  some facts about absolute type from \cite{ABP}
(see also \cite[p.87, Remark (b)]{N2}).

\begin{definition}
\label{def:absolutetype}  Suppose that $\cL$ is a prime perfect fgc Lie algebra.
Then, the centroid $C(\cL)$ of $\cL$ is an integral domain \cite[Lemma 3.3(i)]{ABP}.
Let $\overline{C(\cL)}$ denote  an algebraic closure of the quotient field of
$C(\cL)$, and put $\bar{\cL} = \cL\otimes_{C(\cL)}\overline{C(\cL)}$.
Then, $\bar{\cL}$ is a finite dimensional simple Lie algebra over $\overline{C(\cL)}$
\cite[Prop. 8.7]{ABP}.
We define the \emph{absolute type} of $\cL$ to be the type of this
algebra over $\overline{C(\cL)}$ (that is the type of its root system).
So the type of $\cL$ is one of the reduced types X$_\ell$ in the list~\eqref{eq:type}.
\end{definition}

Absolute type is an isomorphism invariant.  That is, if $\cL$ and $\cL'$
are prime perfect fgc Lie algebras and $\cL$ is isomorphic to $\cL'$, then
$\cL$ and $\cL'$ have the same type \cite[Proposition 8.15]{ABP}.

\begin{proposition}
\label{prop:type}  Suppose that $\cL$ is an fgc centreless Lie $\Lm$-torus,
where $\Lm$ is a free abelian group of rank $n\ge 1$.  Then $\cL$ is
prime and perfect and so its absolute type is defined.  Moreover, if
$\cL$ is bi-isomorphic (or even just isomorphic as a Lie algebra) to a
multiloop Lie torus $\LT(\fs,\bsg,\fh)$, then the absolute type of
$\cL$ is the type of the simple Lie algebra $\fs$ over~$k$.
\end{proposition}

\begin{proof} By the Realization Theorem, we can assume that
$\cL$ is equal to a multiloop Lie torus  $\LT(\fs,\bsg,\fh)$.
The result is now a special case of Theorem 8.16 of \cite{ABP}
(which applies more generally to iterated loop algebras based on $\fs$).
\end{proof}

\subsection{The approach to classification}
\label{subsec:approach}

\emph{For the  rest of the section, suppose that $\fs$ is
a finite dimensional simple Lie algebra
$\fs$ of type~X$_\ell$.}
(We will use $\fs$ as the base algebra for our multiloop
Lie tori.)  Let
\[A = \Aut_k(\fs).\]

\begin{remark}  Let $G$ be a finite abelian subgroup of $A$.

(i)  Let $X(G) = \Hom(G,k^\times)$ be the character group of $G$.
Then we have the \emph{weight space decomposition}
\[\fs = \oplus_{\chi\in X(G)} \fs^\chi\]
of $\fs$ relative to the action of $G$,
where $\fs^\chi = \set{x\in \fs : g(x) = \chi(g) x \text{ for } g\in G}$.
If $1\in X(G)$ is the trivial character (that is
$1(\gm) = 1$ for all $\gm\in\Gm$), then $\fs^1 = \fs^G$, the fixed point
subalgebra of $G$, and each $\fs^\chi$ is an $\fs^G$-module under
the adjoint action.

(ii) We will be interested in finite abelian subgroups $G$ of $A$
that satisfy the following three conditions:
\begin{itemize}
\item[(G1)] $\fg := \fs^G$ is a simple Lie algebra.
\item[(G2)] If $1\ne \chi\in X(G)$ and $\fs^\chi \ne 0$, then
$\fs^\chi \simeq U^\chi \oplus V^\chi$ as a $\fg$-module,
where $\fg$ acts trivially on $U^\chi$ and either $V^\chi = 0$ or
$V^\chi$ satisfies condition (M).
\item[(G3$_n$)] $\mg(G) \le n$.
\end{itemize}
(See Definition \ref{def:condM}.)
Its clear that if $G$ satisfies (G1) (resp.~(G1) and (G2); resp. (G3$_n$)), then any conjugate
of $G$ in $A$ also satisfies (G1) (resp.~(G1) and (G2); resp. (G3$_n$)).

(iii)  Suppose that $\mg(G) \le n$ and so $\gs_n(G) \ne \emptyset$.
Then, $\GL_n(\bbZ)$ acts on $\gs_n(G)$ on the right (see Remark \eqref{rem:actionGL}),
and hence on the left by $P\cdot \bsg = \bsg^{P^{-1}}$.  Also,
the normalizer $N_A(G)$ of  $G$ in $A$ acts on
$\gs_n(G)$ on the left by conjugation:
\[\rho\cdot \bsg  = (\rho\sg_1\rho^{-1},\dots,\rho\sg_n\rho^{-1} )\]
for $\bsg = (\sg_1,\dots,\sg_n)\in \gs_n(G)$ and $\rho \in N_A(G)$.
These left actions commute and so $N_A(G)\times\GL_n(\bbZ)$ acts on the left on $\gs_n(G)$ by
$(\rho,P)\cdot \bsg = (\rho\cdot\bsg)^{P^{-1}} = \rho\cdot(\bsg^{P^{-1}})$.
\end{remark}

We can now state our result on the classification of fgc centreless Lie tori up to bi-isomorphism.

\begin{theorem}
\label{thm:class}
Suppose $n\ge 1$ and $\fs$ is a finite dimensional simple Lie algebra of type X$_\ell$. Set
$A = \Aut_k(\fs)$.
\begin{itemize}
\item[(i)]  Let $\set{G_i}_{i\in I}$ be a nonredundant set of representatives
of the conjugacy classes of the finite abelian subgroups of $A$ that satisfy
(G1), (G2) and (G3$_n$).
\item[(ii)] For $i\in I$, let $\set{\cO(i,j)}_{j\in J_i}$ be the set of orbits
of the action of $N_A(G_i)\times\GL_n(\bbZ)$ on
$\gs_n(G_i)$. (These orbits are unions of the orbits of $\GL_n(\bbZ)$  on
$\gs_n(G_i)$, and the latter orbits were
calculated
in Proposition \ref{prop:orbitrep}.)
\item[(iii)] For $i\in I$ and $j\in J_i$, choose
$\bsg(i,j)\in \cO(i,j)$ with the property that
$\order{G_i} = \order{\bsg(i,j)_1}\dots\order{\bsg(i,j)_n}.$
(This is possible by Proposition \ref{prop:orbitrep}.)
\item[(iv)] For $i\in I$ and $j\in J_i$, let $\fh(i,j)$ be a Cartan subalgebra
of $\fs^{\bsg(i,j)}$.
\end{itemize}
Then, for $i\in I$, $j\in J_i$, $\bsg(i,j)$ satisfies
(A1)--(A3) (in Proposition \ref{prop:N}) and $\LT(\fs,\bsg(i,j),\fh(i,j))$ is a centreless
Lie $\Zn$-torus of absolute type X$_\ell$.  Moreover, if $\Lm$ is a free abelian group of rank $n$,
then any  fgc centreless Lie $\Lm$-torus of absolute type X$_\ell$
is bi-isomorphic to exactly one of the Lie tori $\LT(\fs,\bsg(i,j),\fh(i,j))$.
\end{theorem}

\begin{proof}
For the first statement suppose, that $i\in I$ and $j\in J_i$.  Since $\bsg(i,j)\in \gs_n(G_i)$,
we have
\begin{equation}
\label{eq:Gi}
G_i = \langle \sg(i,j)_1,\dots,\sg(i,j)_n\rangle.
\end{equation}
Hence (A3) in Proposition \ref{prop:N}  holds by the choice of $\bsg(i,j)$ in (iii).
Also, by \eqref{eq:Gi}, the grading of $\fs$ determined by $\bsg(i,j)$
(as in Definition \ref{def:multiloop}) has the same summands
as the weight space decomposition of $\fs$ relative to the action of $G_i$.  So,
since $G_i$ satisfies (G1)  and (G2), $\bsg(i,j)$ satisfies (A1) and (A2).
Thus, $\LT(\fs,\bsg(i,j),\fh(i,j))$ is a centreless Lie $\Zn$-torus by Proposition
\ref{prop:N}.  Moreover, this Lie torus has absolute type X$_\ell$ by Proposition
\ref{prop:type}.

For the second statement, suppose that $\cL$ is an fgc  centreless Lie $\Lm$-torus
of absolute type X$_\ell$, where $\Lm$ is free abelian of rank $n$.
By Theorem \ref{thm:realization}, we can
assume that $\cL = \LT(\fs',\bsg,\fh)$ is a multiloop Lie $\Zn$-torus
based on some finite dimensional simple Lie algebra $\fs'$.  Then,
$\fs'$ has absolute type X$_\ell$ by Proposition \ref{prop:type},
so we can assume that $\fs' = \fs$.  Next
let $G = \langle \sg_1,\dots,\sg_n\rangle$.  Since $\bsg$ satisfies
(A1) and (A2), it follows as in the first part of the proof that
$G$ satisfies (G1) and (G2).  Also, clearly $G$ satisfies (G3$_n$).
Hence, $G$ is conjugate in $A$ to $G_i$ for some $i\in I$.
So, by Theorem \ref{thm:loopbi}, we can assume that $G = G_i$  and thus
$\bsg\in \gs_n(G_i)$.  Hence, $\bsg\in\cO(i,j)$ for some $j\in J_i$.
So $\bsg$ and $\bsg(i,j)$ are both elements of the same orbit $\cO(i,j)$
and hence, by Theorem \ref{thm:loopbi},
$\cL(\fs,\bsg,\fh)$ is bi-isomorphic to $\LT(\fs,\bsg(i,j),\fh(i,j))$.

For the uniqueness part of the second statement, suppose that
the Lie $\Zn$-tori $\LT(\fs,\bsg(i,j),\fh(i,j))$ and
$\LT(\fs,\bsg(i',j'),\fh(i',j'))$ are bi-isomorphic,
where $i,i'\in I$, $j\in J_i$ and $j'\in J_{i'}$.
Then, by Corollary \ref{cor:loopbi1},
$G_i$ is conjugate to $G_{i'}$ in $A$, so
$i' = i$.  Hence, again by Theorem \ref{thm:loopbi},
$\bsg(i,j)$ and $\bsg(i,j')$ are in the same orbit under the action of
$N_A(G_i)\times\GL_n(\bbZ)$ on $\gs_n(G_i)$, so $j = j'$.
\end{proof}

\begin{remark}
\label{rem:class}
Suppose we have the assumptions and notation of Theorem
\ref{thm:class}, and let $i\in I$.

(i) We call the Lie tori $\LT(\fs,\bsg(i,j),\fh(i,j))$, $j\in J_i$,
the \emph{multiloop Lie tori corresponding to $G_i$}.

(ii)  If
either $\mg(G_i) < n$ or the smallest invariant factor of $G_i$ is $\le 4$,
then it follows from Proposition \ref{prop:orbitrep} that $N_A(G_i)\times\GL_n(\bbZ)$
acts transitively on $\gs_n(G_i)$ and so there is only one
multiloop Lie torus corresponding to $G_i$.  This happens frequently in practice.

(iii) If $G_i$ is contained in a torus of $A$, then, by
Corollary \ref{cor:untwisted}, we have $G_i = \set{1}$, so the only multiloop
Lie torus corresponding to $G_i$ is the untwisted one.  On the other hand,
if $G_i$ is not contained in a torus of $A$, the  multiloop
Lie tori corresponding to $G_i$ are not isograded-isomorphic
to the untwisted one (again by Corollary  \ref{cor:untwisted}).
\end{remark}

\begin{example}
\label{ex:F4}
As an example, we obtain the classification of
fgc centreless Lie $\Zn$-tori of absolute type F$_4$, where $n\ge 1$.
For this, let $\fs$ be the finite dimensional
simple Lie algebra of type F$_4$ and let $A = \Aut_k(\fs)$.
We use our results above and the work of C.~Draper and C.~Martin \cite{DM},
who found representatives of the conjugacy classes
of finite abelian subgroups of $A$ that are not contained in a torus.
They labeled these subgroups as I, II.1, II.2, II.3, II.4
and III.

Suppose that $G$ is a finite abelian subgroup of $A$ that satisfies (G1) and (G2),
and suppose that $G$ is not contained in a torus.  If $G$ has label I, II.2, II.3, II.4 or III,
then $G$ does not satisfy both (G1) and (G2).
So $G$ has label II.1.  Then,
$G = \langle \sg_1,\sg_2,\sg_3\rangle \simeq \bbZ_2^3$, where, in the notation of \cite{DM},
\[\sg_1 = \widetilde{\sg_{105}},\quad \sg_2 = t'_{-1,1,-1,1},\quad \sg_3 = t'_{1-,1,-1,1}.\]
Conversely, with this choice of $G$,
 $\fs^G$ is simple of type A$_1$ and, for $1\ne \chi\in X(G)$,
$\fs^\chi \simeq U^\chi \oplus V^\chi$, where $U^\chi$ is 2-dimensional with trivial $\fg$-action
and $V^\chi$ is the (5-dimensional) symmetric module for $\fg$.  So (G1) and (G2) hold.
Note also that $\mg(G) = 3$ and the smallest invariant factor of $G$ is $2$.

So, by Theorem \ref{thm:class} and Remark \ref{rem:class},
the fgc centreless Lie $\Zn$-tori of absolute type
F$_4$ are the untwisted multiloop Lie torus and, if $n\ge 3$, the multiloop Lie torus
$\LT(\fs,\bsg,\fh)$, where $\bsg = (\sg_1,\sg_2,\sg_3,1,\dots,1)$ and
$\fh$ is a Cartan subalgebra of  $\fs^\sg$.
\end{example}

As mentioned in the introduction, to  obtain a classification up to isotopy
of fgc centreless Lie $\Zn$-tori of a given absolute type, one can, at least in principal,
proceed  as follows.  First obtain a classification up to bi-isomorphism
using the approach described in Theorem \ref{thm:class}.
Second, use Theorem \ref{thm:isotopy} and Corollary \ref{cor:isotopy}
to determine which of the representatives of the bi-isomorphism classes are isotopic.
This last step is easy for absolute type F$_4$.

\begin{example}
\label{ex:F4isotopy}
We have carried out the classification of
fgc centreless Lie $\Zn$-tori of absolute type F$_4$ up to bi-isomorphism
in Example  \ref{ex:F4}.  The two representatives found there are
not isotopic by Corollary \ref{cor:isotopy}.
So there are precisely two  fgc centreless Lie $\Zn$-tori of absolute
type F$_4$ up to isotopy.
\end{example}

\end{document}